\documentclass[aop]{imsart}

\RequirePackage{amsthm,amsmath}
\usepackage{xcolor}                   
\usepackage[notext]{stix2}            
\usepackage{microtype}

\definecolor{green}{HTML}{2ECC71}
\definecolor{blue}{HTML}{3498DB}
\definecolor{red}{HTML}{E74C3C}
\definecolor{orange}{HTML}{FD6A02}
\usepackage[pdfpagelabels]{hyperref}  
\hypersetup{
  colorlinks,
  linktocpage,
  urlcolor  = blue,
  citecolor = green,
  linkcolor = red}

\usepackage[shortlabels]{enumitem}

\usepackage{csquotes}

\usepackage[numbers]{natbib}

\usepackage[capitalise,sort]{cleveref} 

\AfterEndEnvironment{proof}{\noindent\ignorespaces} 
\makeatletter
\def\@endtheorem{\endtrivlist}
\makeatother

\Crefname{paragraph}{\S}{\SS}
\Crefname{equation}{}{}
\Crefname{enumi}{}{}
\Crefname{conditioni}{Condition}{Conditions}
\Crefname{conditionalti}{Condition}{Conditions}
\theoremstyle{plain}
\newtheorem{theorem}{Theorem}
\newtheorem*{theorem*}{Theorem}
\Crefname{theorem}{Theorem}{Theorems}

\Crefname{theoremintro}{Theorem}{Theorems}

\newtheorem{lemma}[theorem]{Lemma}
\Crefname{lemma}{Lemma}{Lemmas}
\newtheorem{proposition}[theorem]{Proposition}
\Crefname{proposition}{Proposition}{Propositions}
\newtheorem{corollary}[theorem]{Corollary}
\Crefname{corollary}{Corollary}{Corollaries}

\Crefname{conjecture}{Conjecture}{Conjectures}

\theoremstyle{remark}

\Crefname{example}{Example}{Examples}
\newtheorem*{example*}{Example}
\Crefname{assumption}{Assumption}{Assumptions}
\newtheorem{definition}[theorem]{Definition}
\Crefname{definition}{Definition}{Definitions}

\Crefname{question}{Question}{Questions}
\newtheorem{remark}[theorem]{Remark}
\Crefname{remark}{Remark}{Remarks}

\usepackage{mathtools}
\DeclarePairedDelimiter{\paren}{\lparen}{\rparen}
\DeclarePairedDelimiter{\bracket}{\lbrack}{\rbrack}
\DeclarePairedDelimiter{\set}{\lbrace}{\rbrace}
\DeclarePairedDelimiter{\abs}{\lvert}{\rvert}

\DeclarePairedDelimiter{\norm}{\lVert}{\rVert}
\DeclarePairedDelimiter{\ceil}{\lceil}{\rceil}
\DeclarePairedDelimiterX{\psh}[2]{\langle}{\rangle}{#1, #2}

\DeclarePairedDelimiterXPP{\Exp}[1]{\exp}{\lparen}{\rparen}{}{#1}
\DeclarePairedDelimiterXPP{\Log}[1]{\log}{\lparen}{\rparen}{}{#1}
\DeclarePairedDelimiterXPP{\Inf}[1]{\inf}{\lbrace}{\rbrace}{}{#1}
\DeclarePairedDelimiterXPP{\Sup}[1]{\sup}{\lbrace}{\rbrace}{}{#1}
\DeclarePairedDelimiterXPP{\Max}[1]{\max}{\lbrace}{\rbrace}{}{#1}
\DeclarePairedDelimiterXPP{\Min}[1]{\min}{\lbrace}{\rbrace}{}{#1}

\DeclareMathOperator{\tr}{Tr}
\DeclareMathOperator{\rk}{rk}

\DeclareMathOperator{\esp}{\mathbf{E}}
\DeclareMathOperator{\prob}{\mathbf{P}}
\DeclareMathOperator{\var}{\mathbf{Var}}
\DeclareMathOperator{\law}{\mathbf{law}}
\DeclareMathOperator{\ent}{\mathbf{Ent}}
\DeclareMathOperator{\fish}{\mathbf{I}}

\newcommand{\given}[1][]{%
  \nonscript\:#1\vert
  \allowbreak
  \nonscript\:
\mathopen{}}

\DeclarePairedDelimiterXPP{\Prob}[1]{\prob}[]{}{#1}
\DeclarePairedDelimiterXPP{\Esp}[1]{\esp}[]{}{#1}
\DeclarePairedDelimiterXPP{\Var}[1]{\var}[]{}{#1}
\DeclarePairedDelimiterXPP{\Law}[1]{\law}[]{}{#1}
\DeclarePairedDelimiterXPP{\Ent}[1]{\ent}[]{}{#1}
\DeclarePairedDelimiterXPP{\Fish}[1]{\fish}[]{}{#1}

\arxiv{2303.02628}

\startlocaldefs

\endlocaldefs

\begin{document}

\begin{frontmatter}
\title{Superconvergence phenomenon in Wiener chaoses}
\runtitle{Superconvergence phenomenon in Wiener chaoses}

\begin{aug}
  \author[A]{\fnms{Ronan}~\snm{Herry}\ead[label=e1]{ronan.herry@univ-rennes.fr}\orcid{0000-0001-6313-1372}}
  \author[B]{\fnms{Dominique}~\snm{Malicet}\ead[label=e2]{dominique.malicet@univ-eiffel.fr}\orcid{0000-0003-2768-0125}}
\and
\author[A]{\fnms{Guillaume}~\snm{Poly}\ead[label=e3]{guillaume.poly@univ-rennes.fr}}
\address[A]{IRMAR, Université de Rennes\printead[presep={,\ }]{e1,e3}}

\address[B]{LAMA, Université Guistave Eiffel\printead[presep={,\ }]{e2}}
\end{aug}

\begin{abstract}
We establish, in full generality, an unexpected phenomenon of strong regularization along normal convergence on Wiener chaoses.
Namely, for every sequence of chaotic random variables, convergence in law to the Gaussian distribution is automatically upgraded to \emph{superconvergence}:  the regularity of the densities increases along the convergence, and all the derivatives converges uniformly on the real line.
Our findings strikingly strengthen known results regarding modes of convergence for normal approximation on Wiener chaoses.
Without additional assumptions, quantitative convergence in total variation is established in \cite{nourdin2009stein}, and later on amplified to convergence in relative entropy in \cite{nourdin2014entropy}.

Our result is then extended to the multivariate setting, and for polynomial mappings of a Gaussian field provided the projection on the Wiener chaos of maximal degree admits a non-degenerate Gaussian limit.
While our findings potentially apply to any context involving polynomial functionals of a Gaussian field, we emphasize, in this work, applications regarding:
   improved Carbery--Wright estimates near Gaussianity;
   normal convergence in entropy and in Fisher information;
   \textit{superconvergence} for the spectral moments of Gaussian orthogonal ensembles;
   moments bounds for the inverse of strongly correlated Wishart-type matrices;
   \textit{superconvergence} in the Breuer--Major Theorem.

Our proofs leverage Malliavin's historical idea to establish smoothness of the density via the existence of negative moments of the Malliavin gradient, and we further develop a new paradigm to study this problem.
Namely, we relate the existence of negative moments to some explicit spectral quantities associated with the Malliavin Hessian.
This link relies on an adequate choice of the Malliavin gradient, which provides a novel decoupling procedure of independent interest.
Previous attempts to establish convergence beyond entropy have imposed restrictive assumptions ensuring finiteness of negative moments for the Malliavin derivatives, as in \cite{hu2014convergence,HNTXBreuerMAjor,nourdin2016fisher} or \cite[proposition 5.5]{ledoux2015stein}.
Our analysis renders these assumptions superfluous.

The terminology \emph{superconvergence} was introduced  in \cite{BercoviciVoiculescu} for the central limit theorem in free probability.

\end{abstract}



\end{frontmatter}

\newpage
{\small
\tableofcontents
}
\newpage

\section{Introduction}

\subsection{Summary of the results}
Controlling the regularity of a sequence of asymptotically normal random variables is a prevalent question in probability theory.
In the framework of the usual central limit theorem, the smoothing effect of convolution entails the following regularization phenomenon.
Let $(X_{i})$ be a sequence of centred, normalized, and i.i.d.\ random variables such that $\Esp*{ \mathrm{e}^{\mathrm{i}tX_{1}} } \sim t^{-\theta}$ as $t \to \infty$ for some $\theta > 0$, then for all $q \in \mathbb{N}$, there exists $n$ large enough such that the law of $n^{-1/2} \sum_{i=1}^{n} X_{i}$ has a density with respect to the Lebesgue measure that is $\mathscr{C}^{q}$, and converges in the $\mathscr{C}^{q}$-topology to the Gaussian density;
$\mathscr{C}^{q}$ being the space of functions $f \colon \mathbb{R} \to \mathbb{R}$ with $f, f', \dots, f^{(q)}$ continuous and bounded, equipped with the topology induced by the norm $\norm{f}_{\mathscr{C}^{q}} \coloneq \norm{f}_{\infty} + \dots + \norm{f^{(q)}}_{\infty}$.

Extending normal convergence to non-linear functionals of a random field, in particular, polynomial functionals of a Gaussian field is a fertile and lively area of research.
We refer to \cite{NourdinPeccati} and the references therein, as well as to \cite{NourdinWebpage} for an overview.
Despite numerous results regarding normal approximation, capturing the above regularization phenomenon for Gaussian polynomials has so far remained out of reach: the best known modes of convergence are the total variation distance \cite{nourdin2009stein} or the relative entropy \cite{nourdin2014entropy}.
As opposed to the central limit theorem, thoroughly covered by \cite{LionsToscani}, due to the absence of convolution, questions regarding regularity in this non-linear framework are much more challenging.

In this article, we develop a novel approach to study the regularity of non-linear functionals of a Gaussian field, based on Malliavin calculus and Wiener chaoses theory.
In this setting, we show regularization of densities along normal convergence.
This discovery drastically strengthens the aforementioned results.
Before stating our results, we recall that the Wiener chaoses are the infinite-dimensional counterpart of the well-known Hermite polynomials.
In particular, they form an orthogonal basis with respect to the Wiener measure.
We also recall that non-constant random variables in a finite sum of Wiener chaoses always admit a density with respect to the Lebesgue measure \cite{shigekawa1980derivatives}.
We give more details on Wiener chaoses in \cref{s:wiener}.
We write $d_{\mathrm{FM}}$ for the \emph{Fortet--Mourier} distance; it metrizes the topology of convergence in law.
In the statement below, the Fortet--Mourier distance plays no specific role and could be replaced by any distance metrizing the topology of convergence in law.
  We also write $\mathbb{N}$ for the set of natural integers, and $\mathbb{N}^{*} \coloneq \mathbb{N} \setminus \{0\}$.
Our main result reads as follows.

\begin{theorem}\label{th:main-density}
  Let $d \in \mathbb{N}^{*}$ and $q \in \mathbb{N}$, there exist $\delta = \delta_{q,d} > 0$ and $C = C_{q,d} > 0$, such that for all $F$ in the Wiener chaos of degree $d$, with density $f$, we have:
  \begin{equation*}
    d_{\mathrm{FM}}(F, \mathcal{N}(0,1)) \leq \delta \Rightarrow \bracket*{ f \in \mathscr{C}^{q} \ \text{and} \ \norm{f}_{\mathscr{C}^{q}} \leq C }.
  \end{equation*}
\end{theorem}

  Closely related to \cref{th:main-density}, is the following sequential theorem,
  that gives the announced regularization phenomenon along normal convergence on Wiener chaoses.
Write $\varphi$ for the standard Gaussian density.
\begin{theorem}\label{th:main-sequential}
  Let $(F_{n})$ be a sequence of random variables in a Wiener chaos of fixed degree, with respective density $(f_{n})$.
  Then,
  \begin{equation*}
    F_{n} \xrightarrow[n \to \infty]{law} \mathcal{N}(0,1) \Leftrightarrow \bracket*{\norm{f_{n}^{(q)} - \varphi^{(q)}}_{\infty} \xrightarrow[n \to \infty]{} 0, \qquad \forall q \in \mathbb{N} }.
  \end{equation*}
\end{theorem}
In the above theorem, the quantity $f_{n}^{(q)}$ is only defined for $n$ large enough. 

Our approach to regularity of laws on the Wiener space originates from Malliavin seminal contribution \cite{Malliavin}.
In this paper, Malliavin shows that a random variable $F$ on the Wiener space has a smooth law provided the norm of its \emph{Malliavin derivative} $\Gamma[F,F] \coloneq \norm{\mathsf{D} F}^{2}$ admits negative moments at every order.
Our \cref{th:main-density,th:main-sequential} proceed from Malliavin's strategy together with the following result, that is the pivotal tool of this paper.
\begin{theorem}\label{th:main-negative-moments}
  Let $(F_n)_{n\ge 1}$ be a sequence of random variables in a Wiener chaos of fixed degree.
  Assume that $(F_{n})$ converges in law to $\mathcal{N}(0,1)$, then
\begin{equation*}
  \limsup_{n \to \infty} \Esp*{\Gamma\bracket*{F_n,F_n}^{-q}} < \infty, \qquad q \in \mathbb{N}.
\end{equation*}
\end{theorem}

  By standard Malliavin calculus techniques, which are recalled in \cref{s:malliavin-lemma}, \cref{th:main-density} \& \cref{th:main-sequential}, follows from \cref{th:main-negative-moments}.
  Establishing \cref{th:main-negative-moments} is the main contribution of this paper.
  The proof of \cref{th:main-negative-moments} is conducted in \cref{s:proofs}, and relies on the following key ideas.
  \begin{itemize}
    \item \cref{prop:dimupgrade} shows that \cref{th:main-negative-moments}, stated for scalar random variables, is actually equivalent to its version for vector-valued random variables .
    \item Thanks to a new representation of Malliavin derivatives, we relate, in \cref{s:negative-moments-estimates}, the negative moments of $\Gamma(F,F)$ to spectral quantities associated to the Malliavin Hessian.
    \item For chaos of degree $2$, the Malliavin Hessian is deterministic and the control of these spectral quantities is straightforward (see \cref{s:proof-second-chaos}).
    \item For chaos of degree $\geq 2$, we proceed by induction.
      We first compare the Malliavin Hessian with the \emph{compressed Malliavin Hessian} which is obtained by multiplying it by a large independent Gaussian matrix of small rank, which enables us to reduce the dimension (see \cref{s:compressing}).
      Then, \cref{s:induction} we interpret the compressed Hessian as a vector of random variables in a chaos of degree $d-1$ allowing to conclude by the induction hypothesis.
\end{itemize}
We give a more detailed summary of our approach in \cref{s:scheme}.

\subsection{Compendium of related results}

\subsubsection{Central limit on the Wiener space and the Fourth Moment Theorem}
The breakthrough by Nualart \& Peccati \cite{nualart2005central} provides an efficient and tractable criterion to establish normal convergence on Wiener chaoses.
Their \emph{Fourth Moment Theorem} states that a sequence in a Wiener chaos of fixed degree converges in law to a Gaussian if and only if the sequences of its second and fourth moments converge to the respective moments of the target Gaussian distribution.
This result has stemmed a new line of research establishing simple, yet powerful, conditions for normal convergence on the Wiener space.
Among the most notable developments regarding limit theorems on Wiener chaoses, let us mention the following non-exhaustive contributions.

\begin{itemize}[wide]
  \item[\cite{peccati2005gaussian}] Peccati \& Tudor extends the fourth moment theorem to random vectors whose each coordinate lives in a Wiener chaos, possibly of different degrees.
  \item[\cite{nualart2008central}] Ortiz-Latorre \& Nualart establishes that a sequence of random variables ${(F_{n})}$ in a fixed Wiener chaos converges in law to a Gaussian if and only if $\Gamma[F_{n}, F_{n}]$ converges to a constant in $L^{2}$.
  \item[\cite{nourdin2009stein}] Nourdin \& Peccati combines Stein's method and Malliavin calculus in order to obtain a quantitative fourth moment theorem.
    Namely, for a chaotic random variable $F$ with $\Esp*{F^{2}} = 1$, we have
    \begin{equation*}
      d_{\mathrm{TV}}(F, \mathcal{N}(0,1)) \leq c \Var*{\Gamma[F,F]}^{1/2} \leq c \Esp*{F^{4} - 3}^{1/2}.
    \end{equation*}
    This landmark contribution emphasizes the symbiotic interplay between Stein's method and Malliavin calculus: on the Wiener space, Stein kernels, that quantify convergence in distribution, are explicitly computable through integration by parts for the Malliavin operators.
    See \cite{NourdinWebpage} for a regularly updated list of contributions in this area.
  \item[\cite{ledoux2012chaos,azmoodeh2014fourth}] Ledoux, and Azmoodeh, Campese \& Poly leverage the rich spectral properties of Wiener chaoses to revisit the fourth moment theorem.
    This approach avoids the intricate product formula for Wiener chaoses, and insists instead on moment inequalities for chaotic random variables.
    For further developments of this strategy see \cite{azmoodeh2016generalization,malicet2016squared}
  \item[\cite{nourdin2016fisher}] Nourdin, Peccati \& Swan improve further the Malliavin-Stein approach by establishing a fourth moment bound for the relative entropy with respect to the Gaussian measure.
    In view of Pinsker's inequality, this improves the convergence in total variation of \cite{nourdin2009stein}, although the rate of convergence in \cite{nourdin2016fisher} are non-sharp by an additional logarithmic factor.
\end{itemize}
All the results presented above holds for a \emph{general sequence of chaotic random variables}, that is they hold without any further assumption on the sequence.
For such general sequences, until the present contribution, no results beyond convergence in entropy were available.

\subsubsection{Controlling the regularity via the negative moments of the Malliavin gradient}

Originally, Malliavin \cite{Malliavin} uses controls on negative moments of the Malliavin derivative to give a new, purely probabilistic, proof of Hörmander theorem on hypo-elliptic operators \cite{HormanderHypoelliptic}, see also the recent self-contained survey \cite{hairer2011malliavin}.
Since then, establishing that $\Gamma[F,F]^{-p} \in L^{1}$ has become a practical criterion in the study of the regularity of the density of $F$.
For instance, in various contexts, the recent works \cite{CassFriz,HNTXBreuerMAjor,nourdin2016fisher,AruGMC} implement this strategy.
In this paper, as well as in the companion paper \cite{HMPreg}, we propose a new general estimate on the negative moments of $\Gamma[F,F]$, involving the spectrum of the Hessian matrix of $F$.
We then bring the aforementioned fine results regarding normal convergence on Wiener chaoses, arising from the Malliavin--Stein method, to bear on establishing existence of negative moments for asymptotically normal chaotic sequences.

Previous works on the Wiener space have implemented the strategy of controlling negative moments of the Malliavin derivative to improve normal convergence.
These various attempts fail to capture the generality of the phenomenon we exhibit in this work, and are constrained by unnecessary assumptions in order to carry their analysis.
Let us mention the most prominent developments in that regard.
In the three following examples, the present contribution renders the additional assumptions on the negative moments of the Malliavin derivative unnecessary.
\begin{itemize}[wide]
  \item[\cite{hu2014convergence}] Assuming negative moments for the Malliavin derivative, Hu, Lu \& Nualart give a $\mathscr{C}^{\infty}$ version of the celebrated bound of \cite{nourdin2009stein}.
    Namely, take a sequence $(F_{n})$ of chaotic random variables with variance $1$ and such that
    \begin{equation*}
      \limsup_{n \to \infty} \Esp*{ \Gamma[F_{n}, F_{n}]^{-p} } < \infty, \qquad p \in \mathbb{N},
    \end{equation*}
    they show the following Malliavin--Stein bound for \emph{superconvergence}
    \begin{equation*}
      \norm{f_{n}^{(q)} - \varphi^{(q)}}_{\infty} \leq c_{q} \Esp*{F_{n}^{4} - 3}^{1/2}, \qquad q \in \mathbb{N}.
    \end{equation*}
  \item[\cite{nourdin2016fisher}] In the same spirit, under the assumptions of negative moments, Nourdin \& Nualart establish a fourth moment theorem in relative Fisher information.
     The authors are, moreover, able to apply their criterion to general sequences of random variables living in the \emph{second} Wiener chaos.
    In this case, an explicit diagonalization argument allows to conclude on the existence of negative moments.
    We also refer to \cite[Prop.\ 5.5]{ledoux2015stein} for related bounds on the negative moments of the Malliavin derivatives in connection with the Fisher information.
  \item[\cite{HNTXBreuerMAjor}] Hu, Nualart, Tindel \& Xu establish \emph{superconvergence} to the normal distribution of properly rescaled Hermite sums of a stationary Gaussian field, under the assumption that the spectral measure admits a density whose logarithm is integrable, together with mild additional assumptions on the spectral measure.
\end{itemize}

\subsubsection{Regularity for general polynomials in Gaussian variables}
Our techniques strongly profit from the asymptotic normality of the sequence under consideration.
The question of the regularity of the law for a generic element of a Wiener chaos, possibly away from normality, has attracted several important contributions.
For instance, \cite{BogachevKosovZelenov} establishes that the law of a non-constant polynomials in independent Gaussian variables always belong to a fractional Nilkolskii--Besov space.
Moreover, they show that this regularity is the best possible at this level of generality.
See also the survey \cite{BogachevSurvey}, and the references therein.

\subsubsection{\emph{Superconvergence} in free probability}
In \cite{BercoviciVoiculescu}, Bercovici \& Voiculescu discover a remarkable regularization in the free central limit theorem: indeed for \emph{any} free and identically distributed random variables $(X_{n})$ the law of $n^{-1/2} \sum_{i=1}^{n} X_{i}$ is eventually smooth, and the sequence of respective densities converges to the semi-circular density, in the sense of uniform convergence on compact sets of all the derivatives.
They call this better-than-expected convergence \enquote{\emph{superconvergence}} and we borrow the terminology from their work.

\subsection{Detailed review of the results}

As anticipated, we establish a regularization phenomenon along normal convergence on Wiener chaoses.
Our techniques exploit the rich structure of Wiener chaoses, and yield existence of negative moments for $\norm{\mathsf{D} F_{n}}$, as soon as $(F_{n})$ converges in law to a non-degenerate Gaussian.
This phenomenon has gone unnoticed until now.
It allows, in particular, an important enhancement of the normal convergence on Wiener chaoses: from total variation \cite{nourdin2009stein} or relative entropy \cite{nourdin2014entropy} to \emph{superconvergence}, that is uniform convergence of the densities as well as all their derivatives.
From \cite{nualart2008central,nourdin2009stein}, normal convergence of $(F_{n})$ guarantees that $(\norm{\mathsf{D} F_{n}})$ converges to a constant in $L^p$ ($p \ge 1$).
Here, we refine this information on the behaviour of the Malliavin derivatives and provide analogous results for negative $p$.
This enables Malliavin calculus techniques to establish regularization.
We actually obtain a version of our result for sequences of vectors whose coordinates are in Wiener chaoses, possibly of different degrees, and some variations of the result which hold for finite sums of Wiener chaoses.
We discuss below various applications.

\subsubsection{Regularization on Wiener chaoses}

In what follows, we denote by $\mathcal{W}_m$ the $m$-th Wiener chaos associated to a fixed Gaussian field, and $\Gamma$ the associated \emph{square field} operator, that is $\Gamma[F,F] \coloneq \norm{\mathsf{D} F}^{2}$ where $\mathsf{D}$ is the \emph{Malliavin derivative}.
Whenever, $\vec{F} = (F_{1}, \dots, F_{d})$ is vector-valued, we consider the \emph{Malliavin matrix}
\begin{equation*}
  \Gamma(\vec{F}) \coloneq \paren*{\Gamma[F_{i}, F_{j}]}_{ij}.
\end{equation*}
Following Malliavin's idea, our regularization results are obtained through the existence of negative moments for the Malliavin matrix.
The general version of our theorem for random vectors is as follows.
\begin{theorem}\label{th:main-negative-moments-vector}
  Let $d$ and $m_{1}, \dots, m_{d} \in \mathbb{N}^{*}$.
  Consider a sequence $(\vec{F}_n) \subset \prod_{i} \mathcal{W}_{m_{i}}$ such that
  \begin{equation*}
    \vec{F}_n\xrightarrow[n\to\infty]{\text{Law}}~\mathcal{N}(0,I_d).
  \end{equation*}
  Then, for every $q \ge 1$ there exist $N \in \mathbb{N}$ and $c > 0$ such that
  \begin{equation}\label{eq:seq:thm}
    \Esp*{\det\Gamma (\vec{F}_n)^{-q} }\le c, \qquad n \geq N.
  \end{equation}
\end{theorem}

\begin{remark}
Equivalently the sequence $(\Gamma(\vec{F_n})^{-1})_{n\geq 1}$ is asymptotically bounded in $L^q$, in the sense that, for any matrix norm $\norm{\cdot}$,
\begin{equation*}
  \limsup_{n \to \infty} \Esp*{\norm{\Gamma (\vec{F}_n)^{-1}}^q} < \infty,
\end{equation*}
\end{remark}

  \begin{remark}
    We can consider more general limiting laws $\mathcal{N}(0,C)$ with $C$ an invertible $d \times d$ matrix.
    Without loss of generality, assume that $(m_{i})$ is not decreasing.
    Let $k \leq n$, and $i_{1}, \dots, i_{k}$ be the indices such that $m_{i_{l}} \ne m_{i_{l} +l}$.
    Then, since elements of a Wiener chaos of different degrees are uncorrelated \cite[Prop.\ 2.7.5]{NourdinPeccati}, we find that $C_{n}$, the covariance matrix of $\vec{F_{n}}$ is diagonal by blocks and contains exactly $k$ blocks..
    Write $C_{n}[1], \dots, C_{n}[k]$ for those blocks.
    The block $C_{n}[l]$ is a square matrix of size $n_{l} \times n_{l}$, where $n_{l}$ is the numbers of indices $i$ such that $m_{i} = m_{i_{l}}$.
    Since $C$ is invertible, for $n$ large enough $C_{n}$ is also invertible.
    Inverting the matrix by block, and using that Wiener chaos are stable by non-zero linear combinations, we obtain that $C_{n}^{-1/2} \vec{F}_{n}$ satisfies the assumptions on the theorem.
  \end{remark}

  \begin{remark}
    It is possible to consider a slight generalization of our result, where we consider degrees $m_{i}^{(n)}$ that depend on $n$ but are uniformly bounded.
    In this case the sequence $(m_{i}^{(n)})$ assume only finitely many distinct values.
    Thus, we can extract finitely many sub-sequences, that satisfy the assumptions of the theorem.
    It would be interesting, and more difficult, to consider varying degrees $m_{i}^{(n)}$ with $m_{i}^{(n)} \to \infty$, possibly with some prescribed speed of divergence.
    We do not how to tackle this problem; it would require the demanding and involved task of tracking quantitatively the dependence in the degrees of the chaos in all our estimates.
  \end{remark}

As anticipated, an important consequence of \cref{th:main-negative-moments-vector} is a \emph{superconvergence} phenomenon on Wiener chaoses: by integration by parts, negative moments for $\det \Gamma(\vec{F})$ yield regularity estimates on the density.
  We measure the regularity in the Sobolev space $W^{q,p}(\mathbb{R}^{d})$ with $p \in [1,\infty]$, and $q \in \mathbb{N}$.
  This is the space of (class of equivalence of) functions $f \colon \mathbb{R}^{d} \to \mathbb{R}$ that are in $L^{p}$ such that
  \begin{equation*}
    \partial^{\alpha} f \coloneq \frac{\partial^{q} f}{\partial x_{1}^{\alpha_{1}} \dots \partial x_{d}^{\alpha_{d}}} \in L^{p}, \qquad \alpha = (\alpha_{1}, \dots, \alpha_{d}) \in \mathbb{N}^{d},\, \abs{\alpha} \coloneq \alpha_{1} + \dots + \alpha_{d} = q.
  \end{equation*}
  Here the partial derivatives are taken in the sense of distributions.
  We often omit the dependence on $\mathbb{R}^{d}$ in our notation.
  We equip $W^{p,q}$ with the Sobolev norm
  \begin{equation*}
    \norm{f}_{W^{q,p}} \coloneq \sum_{\abs{\alpha} \leq q} \norm{\partial^{\alpha} f}_{L^{p}}.
  \end{equation*}
  For $q =0$, our definition is understood as $W^{0,p} = L^{p}$.
  Accordingly, we extend our definition of $\mathscr{C}^{q}$ as the set of functions $f \colon \mathbb{R}^{d} \to \mathbb{R}$ with all the partial derivatives $\partial^{\alpha}f$ for $\abs{\alpha} \leq q$ continuous and bounded.
  It is equipped with the norm
  \begin{equation*}
    \norm{f}_{\mathscr{C}^{q}} \coloneq \sum_{\abs{\alpha} \leq q} \norm{\partial^{\alpha}}_{\infty}.
  \end{equation*}

  We recall that by the Sobolev embeddings \cite[\S 4.5]{HormanderAnaFunc}, 
  \begin{align}
    & \label{eq:sobolev-embeddings} \norm{\cdot}_{W^{q,1}} \leq \norm{\cdot}_{W^{q,p}} \leq c \norm{\cdot}_{W^{q',1}}, \qquad p \in [1,\infty],\, q \in \mathbb{N},\ q' \coloneq d + q - \frac{d}{p} \geq q.
  \\& \label{eq:sobolev-embeddings-cq} \norm{\cdot}_{\mathscr{C}^{q}} \leq c \norm{\cdot}_{W^{q',1}}, \qquad q \in \mathbb{N}, \quad q' \coloneq q + d.
\end{align}
Thus, it is sufficient to control the norms $\norm{\cdot}_{W^{q,1}}$ for $q \in \mathbb{N}$, in order to control all the norms $\norm{\cdot}_{W^{q,p}}$ for $q \in \mathbb{N}$ and $p \in [1,\infty]$, and all the norm $\norm{\cdot}_{\mathscr{C}^{q}}$ for $q \in \mathbb{N}$.
To achieve the this control, it is convenient to work in Fourier modes.
We define
\begin{equation*}
  \norm{f}_{N^{q}} \coloneq \sup_{\norm{\vec{t}} \geq 1} \norm{\vec{t}}^{q} \abs{\hat{f}(\vec{t})}, \qquad q \in \mathbb{N}.
\end{equation*}
By the Hausdorff--Young inequality \cite[Thm.\ 7.1.13]{HormanderAnaFunc}, the Fourier isomorphism theorem \cite[Thm.\ 7.1.11]{HormanderAnaFunc}, the continuity of the Fourier transform on $L^{p}$-spaces \cite[Thm.\ 7.9.3]{HormanderAnaFunc}, and \cref{eq:sobolev-embeddings}, we have
\begin{equation}\label{eq:sobolev-embeddings-fourier}
  \norm{\cdot}_{W^{q,1}} \leq c \norm{\cdot}_{N^{q+1}} \leq c' \norm{\cdot}_{W^{q+1,1}}.
\end{equation}

\begin{theorem}\label{th:main-density-vector}
  Let $(\vec{F}_{n})$ be as in \cref{th:main-negative-moments-vector}.
  Then, denoting by $f_n$ the density of $\vec{F}_n$, and by $\varphi$ the density of the standard Gaussian distribution on $\mathbb{R}^{d}$, for every $p \in [1,\infty]$ and $q \geq 0$, we have that for $n$ large enough, $f_n\in W^{q,p}(\mathbb{R}^{d})$ and
  \begin{equation}\label{eq:main:convergence-sobolev}
  f_n \xrightarrow[n\to\infty]{W^{q,p}(\mathbb{R}^{d})} \varphi.
\end{equation}
\end{theorem}

Let us state some immediate consequences and remarks.

\begin{enumerate}[(a),wide]
\item
  In particular, we obtain that $f_n\to \varphi$ in $L^p(\mathbb{R}^{d})$ for every $p \in [1,\infty]$.
  This convergence is already new, except for $p =1$ which is simply total variation.
  By \cref{eq:sobolev-embeddings-cq}, we also deduce that for $n$ large enough, $f_{n} \in \mathscr{C}^{q}(\mathbb{R}^{d})$ and $\norm{f_n\to \varphi}_{\mathscr{C}^{q}}$. 

\item 
  By \cref{eq:sobolev-embeddings-fourier}, \cref{th:main-density-vector} also gives estimates on characteristic functions of vectors in Wiener chaoses close, in law, to a Gaussian vector, which may be advantageous in some circumstances: for every $m \in \mathbb{N}^{*}$ and $q \in \mathbb{N}$ there exist $\delta=\delta_{q,m} > 0 $ and $C=C_{q,m} > 0$ such that  for every $\vec{F} \in\mathcal{W}_m$ we have
\begin{equation}\label{eq:regularity:nonseq:fourier}
  d_{\mathrm{FM}}(\vec{F},\mathcal{N}(0,I_{d}))\le \delta\Rightarrow   \sup_{\norm{\vec{t}} \ge 1} \norm{\vec{t}}^q \ \abs*{\Esp*{\mathrm{e}^{\mathrm{i} \vec{t}\cdot \vec{F}}}}\le C.
\end{equation}

\item
We could in fact be more precise and estimate the rate of convergence.
For example, \cref{th:main-negative-moments-vector} together with \cite[Theorem 4.4]{hu2014convergence} yields, for univariate random variables, that for every $p \in \mathbb{N}$ there exist $C_p$ and $\alpha_p > 0$ such that for every $F \in \mathcal{W}_{m}$ with density $f$
\begin{equation}\label{cor:nualart}
  \Esp*{F^4}-3<\alpha_p \Rightarrow \bracket*{ f \in \mathscr{C}^p\ \text{and}\ \sup_{x\in\mathbb{R}} \norm{f - \varphi}_{\mathscr{C}^{p}} \le C_p \Esp*{F^4-3}^{1/2} }
\end{equation}
This estimate could be generalized to random vectors.
It would be interesting and useful to derive an explicit expression for the quantities $\alpha_p$ and $C_p$.
This rather demanding task falls beyond the scope of this article, and could be explored in further contributions.
\end{enumerate}

  Let us now show how, modulo a well-known result in Malliavin calculus, recalled in \cref{s:malliavin-lemma}, \cref{th:main-negative-moments-vector} implies \cref{th:main-density-vector}.
  Similarly, \cref{th:main-negative-moments} implies \cref{th:main-density,th:main-sequential}.
  \begin{proof}[Proof of \cref{th:main-density,th:main-sequential,th:main-density-vector}]
    By \cref{eq:sobolev-embeddings,eq:sobolev-embeddings-cq,eq:sobolev-embeddings-fourier}, it is sufficient to show that
    \begin{equation*}
      \norm{f_{n} - \varphi}_{N^{q}} \to 0, \quad q \in \mathbb{N}.
    \end{equation*}
    The conclusion of \cref{th:main-negative-moments-vector}, that is \cref{eq:seq:thm}, implies that
    \begin{equation}\label{eq:main:bound-sobolev}
      \limsup_{n \to \infty} \norm{f_{n}}_{W^{q,p}} < \infty, \qquad p \in [1,\infty],\, q \in \mathbb{N}.
    \end{equation}
    For details on this fairly standard result, see \cref{th:malliavin-lemma} below.
    By \cref{eq:sobolev-embeddings-fourier}, \cref{eq:main:bound-sobolev} yields
    \begin{equation}\label{eq:bound-fourier}
      \limsup_{n \to \infty} \norm{f_{n}}_{N^{q}} < \infty, \qquad q \in \mathbb{N}.
    \end{equation}
    Fix $\varepsilon > 0$, and let $A = \varepsilon^{-1}$.
    By the convergence in law, we find that
    \begin{equation*}
      \sup_{1 \leq \norm{\vec{t}} \leq A} \norm{\vec{t}} \norm{\hat{f}_{n}(t) - \hat{\varphi}(t)} \to 0.
    \end{equation*}
    On the other hand, we have
    \begin{equation*}
      \sup_{\norm{\vec{t}} \geq A} \norm{\vec{t}}^{q} \norm{\hat{f}_{n}(t) - \hat{f}(t)} \leq \varepsilon^{q} \norm{f_{n} - \varphi}_{N^{2q}}.
    \end{equation*}
    Since by \cref{eq:bound-fourier}, and smoothness of $\varphi$,
    \begin{equation*}
      c \coloneq \limsup_{n \to \infty} \norm{f_{n} - \varphi}_{N^{2q}} < \infty,
    \end{equation*}
     we find that
     \begin{equation*}
       \limsup_{n \to \infty} \sup_{\norm{\vec{t}} \geq A} \norm{\vec{t}}^{q} \norm{\hat{f}_{n}(t) - \hat{f}(t)} \leq c \varepsilon^{q}.
     \end{equation*}
  This concludes the proof.  
  \end{proof}

\subsubsection{Regularization on sum of chaoses}
The results partially extend to random variables in a finite sum of Wiener chaoses.
We let $\mathcal{W}_{\leq m} \coloneq \bigoplus_{k=0}^m \mathcal{W}_k$. We also denote by $\mathsf{J}_k$ the projection on the $k$-th Wiener chaos.

\cref{th:main-density} does not hold on $\mathcal{W}_{\leq m}$.
Indeed, if $F_n= (n+1)^{-1} G^2+G$ where $G$ belongs to the first chaos of the underlying Gaussian field, then $(F_n)_{n\in \mathbb{N}}$ is a sequence in $\mathcal{W}_{\leq 2}$ converging in law to the standard Gaussian.
Nonetheless, $\Gamma[F_n,F_n]=(2 (n+1)^{-1} G+1)^2$.
It follows that $\Gamma[F_{n}, F_{n}]^{-1}$ is never integrable.
Additionally, by direct computations, the density of $F_n$ is never continuous.
Nevertheless, we obtain regularization results for sequences in $\mathcal{W}_{\leq m}$ under some assumptions on the projection over the largest Wiener chaos $\mathcal{W}_m$.

\begin{theorem}\label{th:main-negative-moments-sum-chaos}
  Fix $m \in \mathbb{N}^{*}$ and a sequence $(F_n) \subset \mathcal{W}_{\leq m}$.
  We assume that:
  \begin{enumerate}[(i)]
    \item $\mathsf{J}_m(F_n)\to \mathcal{N}(0,1)$ in law.
    \item $(F_n)_{n\geq 1}$ is bounded in $L^2$.
  \end{enumerate}
  Then, for every $q\ge 1$ there exist an integer $N$ and a constant $C$ such that 
  \begin{equation*}
    \Esp*{\Gamma\left[F_n,F_n\right]^{-q} }\le C, \qquad n \geq N.
  \end{equation*}
\end{theorem}
Let us make some comments on this theorem.

\begin{itemize}[wide]
  \item 
    The main assumption concerns only the projection of $F_n$ on $\mathcal{W}_m$.
    The projection on the other Wiener chaoses only need to be bounded.
  \item
    The conclusion implies that the density of $F_n$ regularizes as $n$ tends to $+\infty$.
    At this level of generality, $(F_n)_n$ does not converge in law.
    It is not possible to talk about smooth convergence of the densities.
    Nevertheless, by the same argument as in the proof of \cref{th:main-density-vector}, our result implies that all the limits in law of a subsequence of $(F_n)$ have a smooth density, and that the subsequence of the densities converges smoothly.
\end{itemize}

Concerning smooth normal convergence, we state the following corollary.
We write $\mathsf{D}^{2} F$ for the \emph{Malliavin Hessian} of $F$, and for an Hilbert--Schmidt operator $\mathsf{A}$, we write $\rho(\mathsf{A})$ for its spectral radius.
\begin{corollary}\label{cor:densitysumofchaos}
Consider a sequence $(F_n) \subset \mathcal{W}_{\leq m}$ and the sequence of associated densities $(f_n)$.
Assume either of the three following situations hold. 
\begin{enumerate}[(a),wide]
  \item\label{cor:densitysumofchaos:remainder} $F_{n} - \mathsf{J}_{m}F_{n} \to 0$ in $L^{2}$, and $(F_{n})$ converges in law to a standard Gaussian.
  
    \item\label{cor:densitysumofchaos:all-gaussian} For every $k=0,\cdots,m$, $(\mathsf{J}_kF_n)$ converges in law to a Gaussian measure, possibly degenerate except for $k=m$.
  
    \item\label{cor:densitysumofchaos:spectral-radius} $(F_{n})$ is bounded in $L^{2}$, $\rho(\mathsf{D}^{2}F_{n}) \to 0$ in $L^{2}$, and $\liminf \Var*{ \mathsf{J}_{m}F_{n} } > 0$.
\end{enumerate}
Then, we have \emph{superconvergence} to a Gaussian density:
\begin{equation*}
  f_n \xrightarrow[n\to\infty]{W^{q,p}(\mathbb{R})} \varphi, \qquad q \geq 0,\, p \in [1,\infty].
  \end{equation*}
\end{corollary}

\begin{proof}
  The proofs are immediate but for \cref{cor:densitysumofchaos:spectral-radius}.
  Since, for $l \in \set{1, \dots, m}$,
  \begin{equation*}
    \prod_{i=1,i \ne l}^{m} \paren*{\mathsf{L}^{-1} + \frac{1}{i}} = \bracket*{ \prod_{i=1, i \ne l}^{m} \paren*{ \frac{1}{k} - \frac{1}{l}} } \mathsf{J}_{l},
  \end{equation*}
by Meyer's inequalities, (see, for instance \cite[Eq.\ (3.18)]{NourdinPeccatiReinertSecondOrder}), we find that $\rho(\mathsf{D}^{2} \mathsf{J}_{l} F_{n}) \to 0$ in $L^{2}$ and thus also in $L^{4}$ by equivalence of $L^{p}$ norms on Wiener chaos (see \cref{sec:hypercontractivity} below).
  Since $F_{n} \in \mathcal{W}_{\leq m}$, $\Gamma(F_{n},F_{n}) \in \mathcal{W}_{\leq m'}$ for $m' = 2(m-1)$ (see \cite[Prop.\ 2.7.4 \& Thm.\ 2.7.10]{NourdinPeccati}).
  By equivalence of $L^{p}$ norms on Wiener chaos and by \cite[Prop.\ 1.2.2]{nualart2006malliavin}, we find that
  \begin{equation*}
    \Esp*{ \norm{\mathsf{D} F_{n}}^{4}}^{1/4} = \Esp*{ \Gamma(F_{n},F_{n})^{2} }^{1/4} \leq c \Esp*{ \Gamma(F_{n},F_{n})}^{1/2} \leq c' \Esp*{F_{n}^{2}}^{1/2}.
  \end{equation*}
Thus, we are in the setting of the second order Poincaré inequality \cite{chatterjee2009fluctuations,NourdinPeccatiReinertSecondOrder}, and each of the $(\mathsf{J}_{k}F_{n})$ converges in law to Gaussian, possibly degenerate but for $k =m$ in view of our assumption.
  We conclude by \cref{cor:densitysumofchaos:all-gaussian}.
\end{proof}

Actually, from the proof of \cref{th:main-negative-moments-vector}, one could obtain a vector-valued version of our theorem for sum of chaoses.
\begin{theorem}\label{th:main-sum-of-chaos-vector}
  Let $d \in \mathbb{N}^{*}$ and $m_{1}, \dots, m_{d} \in \mathbb{N}^{*}$.
  Consider a sequence $(\vec{F}_{n}) \subset \prod_{i} \mathcal{W}_{\leq m_{i}}$.
  Assume that:
  \begin{enumerate}[(i)]
    \item  $(\mathsf{J}_{m_{1}} F_{n,1}, \dots, \mathsf{J}_{m_{d}} F_{n,d}) \to \mathcal{N}(0, I_{d})$ in law.
  \item $(\vec{F}_n)_{n\ge 1}$ is bounded in $L^{2}$.
\end{enumerate}
Then, for every $q\ge 1$ there exist $N \in \mathbb{N}$ and $C > 0$ such that
  \begin{equation*}
    \Esp*{ \det \Gamma(\vec{F}_{n})^{-q} } \leq C, \qquad n \geq N.
  \end{equation*}
\end{theorem}

\subsection{Scheme of the proof of the main results}\label{s:scheme}
Following Malliavin's idea \cref{th:malliavin-lemma}, the core of the proof is to establish control of negative moments of $\Gamma[F_{n}, F_{n}] = \norm{\mathsf{D} F_{n}}^{2}$ for a sequence $(F_{n}) \subset \mathcal{W}_{m}$ asymptotically normal, that is \cref{th:main-negative-moments}.
Actually, it is sufficient to prove the claim for functionals depending on finitely many independent Gaussian variables $(N_{1}, \dots, N_{K})$ with $K$ arbitrarily large and all the estimates being independent of $K$, as explained in \cref{sec:finitely-generated}.
In this setting, we prove \cref{th:main-negative-moments} by induction on $m$.
The main steps of the proof are as follows.

\paragraph*{Extending the statement to vectors}
Through a discretization procedure, we show, in \cref{cor:dimupgrade}, that the induction hypothesis, that is \cref{th:main-negative-moments} for $m-1$, implies its vectorial version, that is \cref{th:main-negative-moments-vector}, restricted to $(\vec{F}_{n}) \subset \mathcal{W}_{m-1}^{d}$.

\paragraph*{Negative moments for the derivative and spectral remainders of the Hessian}
Our key observation relates the negative moments of $\Gamma[F_{n},F_{n}]$ to spectral quantities associated to the Malliavin Hessian $\mathsf{D}^{2}F_{n}$.
Namely, for all $q \in \mathbb{N}^{*}$, we introduce the spectral quantities 
\begin{equation*}
  \mathcal{R}_{q}(\mathsf{D}^{2}F_{n}) \coloneq \sum_{i_{1} \ne \dots \ne i_{q}} \lambda^{2}_{i_{1}} \dots \lambda^{2}_{i_{q}},
\end{equation*}
where $(\lambda_{i})$ is the spectrum of the random matrix $\mathsf{D}^{2}F_{n}$.
Then, we show, in \cref{th:bound-gamma-spectral-hessian}, that
\begin{equation*}
  \Esp*{ \Gamma[F_{n},F_{n}]^{-q} } \leq c \Esp*{\mathcal{R}_{q'}(\mathsf{D}^{2}F_{n})^{-1/2}}
\end{equation*}
where $q'$ depends only on $q$.
Whenever $m = 2$, that is on the second Wiener chaos, $\mathsf{D}^{2}F_{n}$ is a deterministic matrix, and the above inequality follows from a diagonalization argument together with an explicit computation of the Fourier transform of a chi-squared distribution.
  This step is completed in \cref{th:main-degree-2}.
For the case of higher degree, we use a new decoupling idea based on taking Malliavin derivatives in the direction of Gaussian random variables independent of the underlying field.

\paragraph*{Compressing the Hessian}
Actually, we do not directly derive estimates on $\mathcal{R}_{q}(\mathsf{D}^{2}F_{n})$ but rather control a compressed matrix.
  To do so we generalize, using singular values, the notion of spectral remainders to rectangular matrices (see \cref{def:spectral-remainders-rectangular}).
Then, we show, in \cref{th:rq-compressed,th:choice-compressed}, that the set of matrices $X$ of size $K \times q$ such that negative moments of $\mathcal{R}_{q}(\mathsf{D}^{2}F_{n} \cdot X)$ controls those of $\mathcal{R}_{q}(\mathsf{D}^{2}F_{n})$ is \emph{large}, in a measure-theoretical sense.
The idea is to take $X$ a Gaussian random matrix independent of the underlying field, and show that, this control happens with high probability.

\paragraph*{Control of the compressed Hessian}
We connect the compressed Hessian with the Malliavin matrix of an intermediary random vector living in a Wiener chaos of degree $m-1$.
Namely, define
\begin{align*}
  & \mathsf{D}_{\vec{x}} F_{n} \coloneq \sum_{k=1}^{K} \frac{\partial F_{n}}{\partial N_{k}} x_{k} \in \mathcal{W}_{m-1}, \qquad \vec{x} = (x_{1}, \dots, x_{K}) \in \mathbb{R}^{K};
\\& \mathsf{D}_{X} F_{n} \coloneq (\mathsf{D}_{\vec{x_{1}}}F_{n}, \dots, \mathsf{D}_{\vec{x_{d}}}F_{n}) \in \mathcal{W}_{m-1}^{d}, \qquad X = (\vec{x}_{1}, \dots, \vec{x}_{d}) \in \mathbb{R}^{K \times d}.
\end{align*}
We show that $\mathcal{R}_{q}(\mathsf{D}^{2}F_{n} \cdot X) = \det(\Gamma(\mathsf{D}_{X}F_{n}))$.
Then, we exhibit, in \cref{th:choice-diese-gaussian}, a set of matrices $X$, with large Gaussian measure, such that the law of $\mathsf{D}_{X}F_{n}$ is close to a Gaussian.
To do so, with respect to the enlarged Gaussian field $(N_{k}), (X_{ij})$, we have that $\mathsf{D}_{X}F_{n} \in \mathcal{W}_{m}^{d}$, and we conclude thanks to well-known results linking asymptotic normality on Wiener chaoses and convergence of the norm of the Malliavin derivative to a constant, that are recalled in \cref{s:malliavin-stein}.

\paragraph*{Conclusion}
Since the two sets of matrices constructed in the previous steps have large Gaussian measure, say greater than $2/3$, they have non-empty intersection.
Therefore, our construction yields a matrix $X$ such that the two conditions hold simultaneously.
Since $\mathcal{R}_{q}(\mathsf{D}^{2}F_{n} \cdot X) = \det(\Gamma(\mathsf{D}_{X} F_{n}))$, and $(\mathsf{D}_{X}F_{n}) \subset \mathcal{W}_{m}^{d-1}$, by the vectorial version of the induction hypothesis, we conclude the induction step using that $(\mathsf{D}_{X}F_{n})$ is asymptotically normal.

\paragraph*{Remark on the proof: the importance of Gaussian variables}
Evaluating directional derivatives in independent Gaussian variables plays a decisive role in several steps of the proof.
In this short paragraph, we would like to ease the reader's acclimation to this new paradigm.
The usual Malliavin derivative of a random variable $F$ is defined in the direction of $h \in \mathfrak{H}$, where $\mathfrak{H}$ is an abstract separable Hilbert space.
Due to the isomorphisms between separable Hilbert spaces, the literature has maintained that the choice of $\mathfrak{H}$ is inconsequential.
The present study, together with our companion paper \cite{HMPreg} where we use similar ideas in a non Gaussian setting, puts forward a preferred choice for $\mathfrak{H}$: a Gaussian space, independent of the underlying Gaussian field.
Such choice guarantees that the Malliavin derivative is an element of an enlarged Wiener space, as defined in \cref{def:enlarged-wiener-space}, and allows us to put into action all the fine results regarding the Wiener space.
Historically, we trace back this idea to \cite{bouleau2003error}, where Bouleau chooses $\mathfrak{H}$ to be a copy of the underlying $L^{2}$ space.

\section{Applications}
Our result expresses a broad and versatile phenomenon. 
Numerous statements establish normal convergence for polynomial functionals of a Gaussian field.
Our conclusions potentially comprehend all these situations.
We illustrate the flexibility and the breadth of our analysis with applications coming from different fields, without trying to be exhaustive or stating optimal results.

\subsection{Small ball estimates for multilinear Gaussian polynomials}
The celebrated inequality of Carbery \& Wright \cite{carbery2001distributional} states that, for $(G_1, \dots, G_n)$ a standard Gaussian vector and $P$ a polynomial of degree $d$ such that
  \begin{equation*}
    \Esp*{\abs{ P(G_1, \dots, G_n)}} = 1,
  \end{equation*}
  we have
\begin{equation*}
    \Prob*{\lvert P(G_1,\cdots,G_n)\rvert \le \epsilon} \le c_d\epsilon^{\frac{1}{d}},
\end{equation*}
where $c_d$ depends on $d$ only, and is independent from $n$.
Applied to multilinear homogeneous sums, this inequality plays a crucial role in the seminal contribution \cite{mossel2010noise}.
They obtain quantitative invariance principles in various convergence metrics, and, for the roughest metrics, the resulting bounds may depend on $d$, through the exponents of the maximal influence.
A multilinear homogeneous sum evaluated in a standard Gaussian vector being an archetypal example of Wiener chaos, our \cref{th:main-density-vector} applies.
Thus, provided $d_{\mathrm{FM}}(P(G_1,\cdots,G_n),\mathcal{N}(0,1))$ is small enough, the random variable $P(G_1,\cdots,G_n)$ has a bounded density.
This implies that
\begin{equation*}
  \Prob*{\abs{ P(G_1,\cdots,G_n)} \le \epsilon} \le c_d\epsilon,
\end{equation*}
for an another constant $c_d$.
This considerably improves the exponent on $\epsilon$.

\subsection{Smooth convergence in Breuer-Major Theorem}
Consider $(X_n)_{n\in \mathbb{Z}}$ a stationary sequence of centered and normalized Gaussian variables, and $f \in L^{2}(\gamma)$, where $\gamma \coloneq \mathcal{N}(0,1)$.
Breuer \& Major \cite{breuer1983central} give sufficient conditions for the asymptotic normality of $Z_n \coloneq n^{-1/2} \sum_{k=1}^{n}f(X_k)$.
Define the Hermite rank of $f$ as the smallest integer $s$ such that the projection of $f$ on the $s$-th Hermite polynomial $H_s$ is non-zero.
\cite{breuer1983central} proves that if the correlation function $\rho(k) \coloneq \Esp*{X_0X_k}$ belongs to $\ell^s(\mathbb{N})$ then $(Z_n)$ converges in law to a Gaussian distribution.
In particular, whenever $\rho\in \ell^1(\mathbb{N})$, $(Z_{n})$ converges in law to a Gaussian for any $f \in L^{2}(\gamma)$.
When $f$ is a polynomial, Hu, Nualart, Tindel \& Xu \cite{HNTXBreuerMAjor} gives conditions to ensure $\mathscr{C}^\infty$-convergence of the densities, in terms of logarithmic integrability of the spectral density.
They use their conditions to control the negative moments of the Malliavin derivative of $(Z_n)$.
Since our results provide such controls as soon as we have normal convergence, we obtain that the $\mathscr{C}^\infty$-convergence holds without any additional assumption.
 
\begin{theorem}
Let $(X_n)$ be a stationary normalized Gaussian sequence and $P$ a polynomial.
Assume that the correlation function $\rho$ belongs to $\ell^s(\mathbb{N})$ where $s$ is the Hermite rank of $P$.
Then the density of $Z_n= n^{-1/2}\sum_{k=1}^{n}P(X_k)$ converges to a Gaussian density in $W^{q,p}(\mathbb{R})$ for every $q \ge 0$ and $p \in [1,\infty]$.
\end{theorem}
\begin{proof}
Seeing the variables $(X_k)$ as elements of a Gaussian field, $(Z_n)$ belongs to $\mathcal{W}_{\leq m}$ where $m=\deg(P)$.
Writing $P=\sum_{i=s}^m c_{i} H_{i}$, the projections of $Z_n$ on $\mathcal{W}_{i}$ are given by 
\begin{equation*}
  \mathsf{J}_{i}(Z_n) = \frac{c_{i}}{\sqrt{n}} \sum_{k=1}^{n} H_i(X_k), \qquad i=s,\ldots, m.
\end{equation*}
By \cite{breuer1983central}, each projection $\mathsf{J}_i(Z_n)$ converges in law to a Gaussian variable, non degenerate for $i=m$ since $c_m\not=0$.
The result follows from \cref{cor:densitysumofchaos} \cref{cor:densitysumofchaos:all-gaussian}.
\end{proof}

\subsection{Normal convergence in entropy and Fisher information}
Let us recall some notions from information theory.
Let $\varphi$ be the density of the standard Gaussian distribution on $\mathbb{R}^{d}$.
Let $\vec{F}$ be a random vector of $\mathbb{R}^d$ with density $f$.
The \emph{relative entropy} of $\vec{F}$ with respect to $\mathcal{N}(0,I_{d})$ is 
\begin{equation*}
  \Ent{\vec{F}} \coloneq \Esp*{ \log  \frac{f(\vec{F})}{\varphi(\vec{F})} } = \Esp*{ \log f(\vec{F}) } + \frac{1}{2} \Esp*{ \norm{\vec{F}}^{2} } + \frac{d}{2} \log 2\pi;
\end{equation*}
while its \emph{relative Fisher information} is 
\begin{equation*}
\Fish{\vec{F}} \coloneq \Esp*{ \norm*{\vec{\nabla} \log \frac{f(\vec{F})}{\varphi(\vec{F})} }^2 } = \int_{\mathbb{R}^d}\frac{\norm{\vec{\nabla}f(x)}^2}{f(x)} \mathrm{d}x - \Esp[\big]{\norm{\vec{F}}^{2}}.
\end{equation*}
The total variation distance, the relative entropy, and the relative Fisher information are related through Pinsker's inequality, and the log-Sobolev inequality \cite{GrossLogSob}
\begin{equation*}
  d_{\mathrm{TV}}(\vec{F}, \mathcal{N}(0,I_{d}))^{2} \leq \frac{1}{2} \Ent{\vec{F}} \leq \frac{1}{4} \Fish{\vec{F}}.
\end{equation*}
Thus, convergence in Fisher information is an improvement to the convergence in entropy, which is itself an improvement to the convergence in total variation.
Finally, we define the multivariate \emph{score function} of $\vec{F}$
\begin{equation*}
  \vec{\rho} = (\rho_1,\ldots,\rho_d) \coloneq \vec{\nabla} \log f.
\end{equation*}
In this way,
\begin{equation*}
\Fish{\vec{F}} = \Esp[\big]{ \norm{\vec{\rho}(\vec{F})}^2 - \norm{ \vec{F} }^{2} } = \sum_{i=1}^d \Esp[\big]{ \rho_i(\vec{F})^2 - F_{i}^{2} }.
\end{equation*}
The following integration by parts characterises the score function:
\begin{equation}\label{eq:score-ipp}
  \Esp*{ \partial_i \Phi(\vec{F}) }=\Esp*{ \rho_i(\vec{F})\Phi(\vec{F}) }, \qquad \Phi \in \mathscr{C}_c^1(\mathbb{R}^{d}),\, i = 1,\dots, d.
\end{equation}

Consider a sequence of isotropic vectors $(\vec{F}_n) \subset \mathcal{W}^{d}_{m}$ converging in law to $\vec{N}$, the standard Gaussian vector on $\mathbb{R}^{d}$.
We recall that isotropic means that the covariance matrix of $\vec{F_{n}}$ is $I_{d}$ for all $n \in \mathbb{N}$.
Let $f_n$ be the density of $\vec{F}_n$.
By \cite{nourdin2014entropy}, we have that $\Ent{\vec{F}_{n}} \to 0$.
More precisely, they show the bound:
\begin{equation*}
  \Ent{\vec{F}_{n}} \leq O(\Delta_n|\log \Delta_n|), \qquad \Delta_n \coloneq \Esp*{\norm{\vec{F_{n}}}^{4} - \norm{\vec{N}}^{4} }.
\end{equation*}
\cite{nourdin2014entropy} actually provides an analogous bound for non isotropic random vectors.
We focus on the isotropic case for simplicity.
The general case can be obtained by multiplying all the $\vec{F}_{n}$ by the square root of the inverse of their covariance matrix.
This bound is sub-optimal since by \cite{ArtsteinBallBartheNaorEntropicCLT,JohnsonBarronEntropicCLT,BobkovChistyakovGotzeEntropicCLT}, in the case of sums of i.i.d., centred, and normalized random variables $S_n=n^{-1/2}\sum_{k=1}^n X_k$, we have
\begin{equation*}
  \Fish{S_{n}} \leq O(n^{-1}).
\end{equation*}
Our findings allow us to improve upon the results of \cite{nourdin2014entropy}, and to provide an optimal rate of convergence in entropy on Wiener chaoses.
Actually, we obtain directly an optimal rate of convergence in Fisher information.

\begin{theorem}\label{th:entropy-vector}
  Fix $d \in \mathbb{N}^{*}$, and $m_1,\ldots,m_d \in \mathbb{N}^{*}$ and a sequence of isotropic random vectors $(\vec{F}_n)_{n\in \mathbb{N}} \subset \prod_{i} \mathcal{W}_{m_{i}}$.
  Assume that $(\vec{F}_{n})$ converges in law to the standard $d$-dimensional Gaussian distribution $\vec{N}$. Then there exists a constant $C$ such that for $n$ in $\mathbb{N}$ large enough
  \begin{equation*}
    \Ent{\vec{F}_{n}} \leq \frac{1}{2} \Fish{\vec{F}_{n}} \leq C \Delta_n, \qquad \Delta_n \coloneq \Esp*{\norm{\vec{F}_n}^4 - \norm{\vec{N}}^4}.
  \end{equation*}
\end{theorem}
  Actually from this result we can obtain a uniform bound for the relative entropy in the case $d=1$.

\begin{corollary}\label{cor:entropy-uniform}
 Fix $m \in \mathbb{N}^{*}$.
 There exists a constant $C=C_{m}$ such for any $F\in \mathcal{W}_{m}$ with unit variance, we have 
\begin{equation*}
  \Ent{F} \leq C\Esp*{F^4 - 3}.
\end{equation*}
\end{corollary}

 \begin{remark}
   Thanks to Pinsker inequality, \cref{cor:entropy-uniform} implies the celebrated inequality \cite{nourdin2009stein}
\begin{equation*}
  d_{\mathrm{TV}}(F, \mathcal{N}(0,1))^{2} \leq C \Esp*{F^4 - 3}, \qquad F\in \mathcal{W}_{m} \ \esp F^{2} = 1.
\end{equation*}
However, in the inequality of \cite{nourdin2009stein}, the constant $C$ does not depend on the order of the chaos $m$.
For instance, \cite{ChenPoly} gives $C = 1/3$.
\end{remark}

\begin{proof}[Proof of \cref{cor:entropy-uniform}]
  Since $\esp \Gamma(F, F) = m$, arguing as in \cite[Prop.\ 4.2]{NoudinNualartPoly}, one can find $p = p_{m} > 1$ such that the $L^{p}$-norm of the density of $F$ is uniformly bounded, thus there exists a constant $C>0$ depending only on the $m$ such that $\Ent{F}\leq C$, uniformly in $F \in \mathcal{W}_{m}$ with unit variance.
  \cref{th:entropy-vector} gives $\delta>0$ and $C'>0$ such that for any such $F$, if $\Delta \coloneq \Esp*{F^4 -3} < \delta$, then $\Ent{F} \leq C'\Delta $.
  These two observations implies the claim.
\end{proof}

\begin{remark}
  Let us comment on the extension of \cref{cor:entropy-uniform} to the multivariate case.
  Following \cite{NoudinNualartPoly}, a uniform bound $\Esp{ \det \Gamma(\vec{F})} \geq \beta$ for some $\beta > 0$, yields a uniform upper bound for $\Ent{\vec{F}}$.
  If we assume such bound, we could implement the same strategy as above.
  We stress that the mere assumptions $\vec{F} \in \prod_{i} \mathcal{W}_{m_{i}}$ isotropic does not imply $\Esp{ \det \Gamma(\vec{F})} > 0$.
\end{remark}
The proof of \cref{th:entropy-vector} relies on two lemmas.
The first lemma is a consequence of our main result.
\begin{lemma}\label{th:bound-inverse-malliavin-matrix}
  Let $d \in \mathbb{N}^{*}$, and $m_1,\ldots,m_d \in \mathbb{N}^{*}$.
  Consider a sequence of isotropic vectors $(\vec{F}_n) \subset \prod_{i} \mathcal{W}_{m_{i}}$.
 Let $\vec{N}\sim \mathcal{N}(0,I_d)$.
 Assume that
 \begin{equation*}
   \vec{F}_n\xrightarrow[n\to\infty]{\text{Law}}~\vec{N}.
 \end{equation*}
 Define
 \begin{align*}
  & W_n \coloneq (W_n^{ij})_{1\leq i,j\leq d} \coloneq \Gamma(\vec{F}_n)^{-1};
\\& W \coloneq (W^{ij})_{1\leq i,j\leq d} \coloneq \mathrm{diag}(m_1^{-1},\ldots,m_d^{-1}).
 \end{align*}
 Then, for every $p \geq 1$ there exist $C > 0$ and $N > 0$ such that for every $n\ge N$ and every $1\leq i,j\leq d$:
\begin{equation*}
  \|W_n^{ij}-W^{ij}\|_{L^p} + \|\Gamma[W_n^{ij},W_n^{ij}]\|_{L^p}\leq C \Delta_n^{1/2}, \qquad \Delta_n \coloneq \Esp*{ \|\vec{F}_n\|^4 -  \|\vec{N}\|^4 }.
\end{equation*}
\end{lemma}
\begin{proof}
  We fix a matrix norm $\norm{\cdot}$, and $p \geq 1$.
  If $M$ is a random matrix, we write
  \begin{equation*}
    \norm{M}_{L^p} \coloneq \Esp*{ \norm{M}^p }^{\frac{1}{p}}.
  \end{equation*}
As a consequence of \cref{th:main-negative-moments-vector}, there exists a constant $C > 0$ such that for $n$ large enough,
\begin{equation*}
  \norm{\Gamma(\vec{F}_n)^{-1}}_{L^{2p}}\leq C.
\end{equation*}
By hypercontractivity \cref{eq:hypercontractivity}, and then \cite[Proof of Thm.\ 4.2]{NourdinRosinski} there exists $C>0$ such that for any $n\ge 1$
\begin{equation*}
  \norm{\Gamma(\vec{F}_n)-D}_{L^{2p}} \leq C_{0} \norm{\Gamma(\vec{F}_{n}) - D}_{L^{2}} \leq C \Delta_n^{1/2},
\end{equation*}
where $D \coloneq \mathrm{diag}(m_1,\ldots,m_d)$.
  Using that
  \begin{equation*}
    \Gamma(\vec{F}_n)^{-1}-D^{-1}=\Gamma(\vec{F}_n)^{-1}(D-\Gamma(\vec{F}_n))D^{-1},
  \end{equation*}
  we deduce by continuity of the matrix product and the Cauchy--Schwarz inequality that there exists $C > 0$ such that for $n$ large enough,
\begin{equation*}
  \norm{ \Gamma(\vec{F}_n)^{-1}-D^{-1} }_{L^p}\leq C \Delta_n^{1/2}.
\end{equation*}
Equivalently, $\norm{W_n-W}_{L^p}\leq C \Delta_n^{1/2}$, thus, $\sup_{i,j}\|W_n^{ij}-W^{ij}\|_{L^p}\leq C \Delta_n^{1/2}$ with possibly another constant $C > 0$.
This shows the inequality for the first term in the left-hand side.

For the second term, by Cramer formula
\begin{equation*}
  W_n^{ij}-W^{ij} = \frac{Z_n^{ij}}{\det(\Gamma(\vec{F}_n))},
\end{equation*}
where $Z_n^{ij}$ is polynomial in the entries of $\Gamma(\vec{F}_n)$.
We have $\norm{Z_n^{ij}}_{L^{2p}}\leq C \Delta_n^{1/2}$, and since $Z_{n}^{ij}$ is polynomial, we deduce, by integration by parts, the bound $\norm{ \Gamma[Z_n^{ij},Z_n^{ij}] }_{L^{2p}} \leq  C \Delta_n^{1/2}$.
These two bounds and the bound on negative moments of $\det(\Gamma(\vec{F}_n))$ give the bound $\norm{\Gamma[Z_n^{ij},Z_n^{ij}]}_{L^p}\leq  C \Delta_n^{1/2}$ for $n$ large enough.
\end{proof}

The second lemma provides, on Wiener chaoses, an explicit formula for the score function through integration by parts with Malliavin operators.
\begin{lemma}\label{th:score-vector}
  Let $d \in \mathbb{N}^{*}$, $m_1,\ldots,m_d \in \mathbb{N}^{*}$.
  Take $\vec{F} = (F_{1}, \dots, F_{d}) \in \prod_{i} \mathcal{W}_{m_{i}}$.
  Define $W = (W_{ij})_{1\leq i,j\leq d} \coloneq \Gamma(F)^{-1}$.
  Then, the score function $\vec{\rho}=(\rho_1,\ldots,\rho_d)$ of $\vec{F}$ is given by
\begin{equation*}
  \rho_i(\vec{F}) \coloneq \sum_{j=1}^d\Esp*{ m_jF_jW_{ij}-\Gamma[F_j,W_{ij}] \given \vec{F} }.
\end{equation*}
\end{lemma}
\begin{proof}
  Fix $\Phi\in \mathscr{C}_c^1(\mathbb{R}^d)$.
  By the chain rule for $\Gamma$ 
  \begin{equation*}
\Gamma[\Phi(\vec{F}),F_i]=\sum_{j=1}^d \partial_j\Phi(\Vec{F}) \Gamma[F_i,F_j], \qquad i = 1, \dots, d.
\end{equation*}
   In matrix notation,
\begin{equation*}
  \begin{pmatrix}
    \Gamma[\Phi(\vec{F}),F_1]\\ \vdots \\ \Gamma[\Phi(\vec{F}),F_d]
    \end{pmatrix}
    = \Gamma(\vec{F})
    \begin{pmatrix}
    \partial_1\Phi(\vec{F})\\ \vdots \\ \partial_d\Phi(\vec{F})
  \end{pmatrix}.
\end{equation*}
Equivalently,
\begin{equation*}
  \begin{pmatrix}\partial_1\Phi(\vec{F})\\ \vdots \\ \partial_d\Phi(\vec{F})\end{pmatrix}=W\begin{pmatrix}\Gamma[\Phi(\vec{F}),F_1]\\ \vdots \\ \Gamma[\Phi(\vec{F}),F_d]\end{pmatrix}.
\end{equation*}
Thus, we find
\begin{equation*}
  \partial_i\Phi(\vec{F})=\sum_{j=1}^dW_{ij}\Gamma[\Phi(\vec{F}),F_j], \qquad i =1,\dots, d.
\end{equation*}
By integration by parts, we deduce
\begin{equation*}
  \begin{array}{ll}\Esp*{ \partial_i\Phi(\vec{F}) }&=\sum_{j=1}^d\Esp*{ W_{ij}\Gamma[\Phi(\vec{F}),F_j] }\\
&=\sum_{j=1}^d \Esp*{ W_{ij}\Phi(\vec{F})m_jF_j }-\Esp*{ \Gamma[F_j,W_{ij}]\Phi(\vec{F}) }\\
&=\sum_{j=1}^d \Esp*{ (m_jF_jW_{ij}-\Gamma[F_j,W_{ij}])\Phi(\vec{F}) }\end{array}
  \end{equation*}
  The result follows in view of \cref{eq:score-ipp}.
\end{proof}

\begin{proof}[Proof of {\cref{th:entropy-vector}}]
  For the sake of conciseness, we drop the dependence in $n$ in this proof.
  By \cref{th:score-vector}, we have
  \begin{equation*}
    \rho_{i}(\vec{F}) - F_{i} = \Esp*{ Z_{i} \given \vec{F} },
  \end{equation*}
  where
  \begin{equation*}
    Z_{i} = (m_{i} - W_{ii}^{-1}) W_{ii}F_{i} + \sum_{j \ne i} m_jF_jW_{ij}-\Gamma[F_j,W_{ij}].
  \end{equation*}
  By \cref{th:bound-inverse-malliavin-matrix}, we find that
  \begin{equation*}
    \norm{Z_{i}} \leq C \Delta^{1/2}.
  \end{equation*}
  In this case, we conclude by the triangle inequality.

\end{proof}

\subsection{Regularization of spectral moments of random matrices}
In linear algebra, many quantities of interest, such as moments of the spectral measure, are polynomials in the entries of the matrix.
Thus, the theory of random matrices provides another context where our results naturally apply.

Let $n \in \mathbb{N}^{*}$.
The \emph{Gaussian Orthogonal Ensemble} $GOE(n)$ is the probability distribution on the set of $n \times n$ symmetric matrices with density with respect to the Lebesgue measure is proportional to $\Exp*{ - n \tr(A^{2})/4 }$.
Equivalently a random $n \times n$ symmetric matrix $A_{n} \sim GOE(n)$ if and only if the entries of $A_{n}$ above the diagonal are independent Gaussian, with variance $n^{-1}$ out of the diagonal, and with variance $2n^{-1}$ on the diagonal.
Following the famous semicircle law \cite{WignerSemiCircle}, when properly rescaled, the moments of the spectral measure of $A_n$ converges to the respective moments of the semicircle law:
\begin{equation*}
  \frac{1}{n} \tr A_{n}^{p} \xrightarrow[n \to \infty]{a.s.} c_{p} \coloneq \frac{1}{2\pi} \int_{-2}^{2} x^{p} (4 - \abs{x}^{2})^{1/2} \mathrm{d} x.
\end{equation*}
Moreover, by \cite{JohanssonRandomMatrix}, the normalized fluctuations $\tr(A_n^p)-nc_p$ converge in distribution to a Gaussian limit $\mathcal{N}(0,\sigma_{p}^{2})$ for some $\sigma_p \ne 0$ .

We stress that both Wigner \cite{WignerSemiCircle} and Johansson \cite{JohanssonRandomMatrix} results are actually available for symmetric random matrices with entries possibly non Gaussian.
In the case of Gaussian entries, our results improve the mode of convergence of the fluctuations.

\begin{theorem}\label{Wigner}
  Let $A_n \sim GOE(n)$, for each $n \in \mathbb{N}^{*}$, and let $p\geq 1$.
  Then the sequence of densities of $\tr(A_n^p)-nc_p$,
  converges to a Gaussian density in $W^{q,r}(\mathbb{R})$ for every $q\ge 0$ and $r \in [1,\infty]$.
\end{theorem}
\begin{proof}
  For the convergence of densities in Sobolev spaces, we use \cref{cor:densitysumofchaos} \cref{cor:densitysumofchaos:spectral-radius}.
  Consider the Gaussian field on $\mathbb{N}^{*} \times \mathbb{N}^{*}$ of independent standard Gaussian $(G_{i,j})$.
  Write $F_{n} \coloneq \tr A_{n}^{p} - n c_{p}$.
  We have that
  \begin{equation*}
    \tr A_n^p = \frac{1}{n^{\frac{p}{2}}}\sum_{1\leq i_1,\ldots,i_p\leq n} G_{i_1,i_2}G_{i_2,i_3}\cdots G_{i_{p-1},i_p} G_{i_p,i_1}.
  \end{equation*}
  Thus $F_{n} \in \mathcal{W}_{\leq p}$.
  \cite{chatterjee2009fluctuations} shows that $\rho(\mathsf{D}^{2} F_{n}) \to 0$.
    Moreover, $(F_{n})$ is bounded in $L^{2}$.
  We are left to verify that $\Var*{\mathsf{J}_{p} F_{n}} \geq O(1)$.
  Let us write
  \begin{equation*}
    \mathscr{I}_{p} \coloneq \set*{ (i_{1}, \dots, i_{p}) \in \{1,\dots,n\}^{p} : \{i_{l}, i_{l+1}\} \ne \{i_{l'}, i_{l'+1}\}, \, l =1, \dots, p }.
  \end{equation*}
  In view of the explicit expression of $F_{n}$, we find that
  \begin{equation*}
    \mathsf{J}_{p} F_{n} = \frac{1}{n^{p/2}} \sum_{(i_{1}, \dots, i_{p}) \in \mathscr{I}_{p}} G_{i_1,i_2}G_{i_2,i_3}\cdots G_{i_{p-1},i_p} G_{i_p,i_1} + R_{n}.
  \end{equation*}
  In the expression above, in view of the definition of $\mathscr{I}_{p}$, all the random variables appearing in the sum are independent.
  The term $R_{n}$, whose explicit expression is irrelevant, consists in sums of degree $p$ of products of Hermite polynomials evaluated in the $G_{ij}$'s, at least one of these polynomials being of degree strictly greater than $1$.
  By independence of the $G_{ij}$'s, the first sum and $R_{n}$ are uncorrelated.
  Thus
  \begin{equation*}
    \begin{split}
      \Var*{ \mathsf{J}_{p} F_{n}} & \geq \frac{1}{n^{p}} \Esp*{\paren*{ \sum_{(i_{1}, \dots, i_{p}) \in \mathscr{I}_{p}} G_{i_1,i_2}G_{i_2,i_3} \dots G_{i_{p-1},i_p} G_{i_p,i_1} }^{2} }
                                \\ &\geq \frac{c_{p}}{n^{p}} \abs*{ \set*{ (i_{1}, \dots, i_{p}) \in \{1,\dots,n\}^{p}  : i_{1} < i_{2} < \dots < i_{p} } }
                                 \\& = \frac{c_{p}}{n^{p}} \binom{n}{p} \geq O(1).
    \end{split}
  \end{equation*}
  \cref{cor:densitysumofchaos} gives the announced convergence in $W^{q,r}(\mathbb{R})$ for every $q \leq 0$ and $r \in [1,\infty]$.
\end{proof}

\begin{remark}
  We stress that our approach is rather general and could be extended to a situation where the $(G_{i,j})$ have more general variances, that is $c_{1} n^{-1} \leq \Var*{G_{i,j}} \leq c_{2} n^{-1}$ for some constants $c_{1},\, c_{2} > 0$.
\end{remark}

\subsection{Control of the inverse of strongly correlated Wishart-type matrices}

Fix $n$ and $d \in \mathbb{N}^{*}$.
Let $B$ be a $n\times d$ matrix whose lines are independent random vectors of $\mathbb{R}^d$ with common distribution $\mathcal{N}(0,\Sigma)$.
\emph{Wishart matrices} are, in their classical sense, matrices of the form $A={}^tB B$.
We can see the lines of $B_n$ as realizations of normal experiments, and see $A_n \coloneq \frac{1}{n} {}^{t}B_{n} B_{n}$ as the empirical covariance matrix of the sample.
When $p$ is fixed and $n \to +\infty$, after renormalization the sequence of Wishart matrices converges to the actual covariance:
\begin{equation*}
  A_n \to \Sigma.
\end{equation*}
We consider a broad generalization of Wishart matrices.
The lines of $B_n$ are not necessarily independent, nor identically distributed, and we only assume the convergence property $A_n\to \Sigma$.
We then obtain a good control on the inverse $A_n^{-1}$.
Our general version allows for correlation in the sample and might be of interest in statistics.
\begin{theorem}\label{WishartInverse}
Fix $n$ and $p \in \mathbb{N}^{*}$.
Let $A_n= {}^{t}B_nB_n$ where $B_n$ is a matrix $n\times p$ with entries in a Gaussian field. We assume that $(A_n)$ converges in probability to a deterministic invertible matrix $\Sigma$.
Then for every $q\geq 1$, there exist an integer $N$ and a constant $C$ such that
\begin{equation*}
  \Esp*{\det (A_n)^{-q}}\leq C, \qquad n \geq N.
\end{equation*}
In particular, $A_n^{-1}\to \Sigma^{-1}$ in $L^q$ for every $q\ge 1$.
\end{theorem}

\begin{proof}
  Consider a Gaussian vector $\vec{G} = (G_{1}, \dots, G_{d}) \sim \mathcal{N}(0, I_{d})$.
  Without loss of generality, assume that the $G_{i}$'s are elements of the underlying Gaussian field, but independent of the entries of $A_{n}$.
  Let $\vec{F}_n \coloneq B_n \vec{G} \in \mathcal{W}_2^d$.
  Conditionally to $B_n$, $\vec{F}_n$ is a Gaussian vector with covariance matrix ${}^{t}B_nB_n=A_n$.
  Since, by assumption, $A_{n} \to \Sigma$ in probability, we deduce that $\vec{F}_n$ converges in distribution to $\mathcal{N}(0,\Sigma)$.
  By \cref{th:main-negative-moments-vector}, for $n$ large enough,
  \begin{equation*}
    \Esp*{\det(\Gamma(\vec{F}_n))^{-q}} \leq c.
  \end{equation*}
    We have $F_{i} = \sum_{k} B_{n}[i,k] G_{k}$, thus by bi-linearity and independence of $B_{n}$ and $\vec{G}$
  \begin{equation*}
    \begin{split}
      \Gamma(F_{i}, F_{j}) &= \sum_{k,l} \Gamma(B_{n}[k,i], G_{k}, B_{n}[l,j] G_{j}) 
                         \\&= \sum_{k,l} G_{k} G_{l} \Gamma(B_{n}[k,i], B_{n}[l,j]) + \sum_{k,l} B_{n}[k,i] \Gamma(G_{i}, G_{j}) B_{n}[l,j].
    \end{split}
  \end{equation*}
  Defining $M$ is the non-negative random matrix, whose entries are given by
  \begin{equation*}
    M_{ij} \coloneq \sum_{k,l} G_{k} \Gamma(B_{n}[i,k], B_{n}[j,l]) G_{l},
  \end{equation*}
  we thus have, using that $\Gamma(\vec{G}) = I_{d}$,
  \begin{equation*}
    \Gamma(\vec{F}_{n}) = {}^{t}B_{n} B_{n} + M = A_{n} + M,
  \end{equation*}
This give $\Gamma(\vec{F}_{n}) \geq A_{n}$.
The conclusion follows.
\end{proof}

\begin{remark}
  \begin{enumerate}[(a), wide]
    \item 
  For simplicity, our statement is formulated for a matrix $B_{n}$ with entries taking values in a Gaussian field.
  From the proof, we see that the conclusion of the theorem remains valid if we take the entries of $B_{n}$ with values in a Wiener chaos $\mathcal{W}_{m}$.

\item
  Whenever the lines of $B_n$ are i.i.d., developing the determinant with the Cauchy--Binet formula and bounding it from below by a sum of positive independent terms yield a more direct proof.

\item
  We rely on similar strategy to obtain explicit estimates in the proof of our main theorems.
\end{enumerate}
\end{remark}

\section{Prolegomena on Wiener chaoses}\label{s:wiener}

In this section, we provide the necessary definitions, notations, and preliminary lemmas required for the proof of the main theorem.
In all the article, for any parameter $\alpha$, $C_\alpha$ stands for a constant which only depends on $\alpha$, and whose value may possibly change from line to line.
Nevertheless, for the sake of clarity, we generally do not track the dependence on the order of the chaoses, typically denoted by $m$.

\subsection{Succinct review on Wiener chaoses}\label{s:wiener-chaos}

\subsubsection{The Wiener space}

We let $\gamma \coloneq \mathcal{N}(0,1)$ be the standard Gaussian distribution on $\mathbb{R}$.
We work on the following countable product of probability spaces which we call \emph{a Wiener space}:
\begin{equation}\label{def:wiener-space}
  (\Omega,\mathscr{F},\prob) \coloneq \left(\mathbb{R},\mathscr{B}(\mathbb{R}),\gamma\right)^{\mathbb{N}}.
\end{equation}

We define the \emph{coordinate functions}
\begin{equation*}
N_i \coloneq
\begin{cases}
  \Omega & \longrightarrow  \mathbb{R},
  \\ (x_{0}, x_{1},\cdots) & \longmapsto  x_i.
\end{cases}
\end{equation*}
By construction, the $N_i$'s are independent random variables on $(\Omega,\mathscr{F},\prob)$, with common law the standard Gaussian distribution.
We sometimes require countably many auxiliary independent standard Gaussian random variables, say $(N_i')$, independent of $(N_i)$.
In the same way, we build such a family as coordinates of auxiliary Wiener spaces and one is left to work for instance on an enlarged Wiener space

\begin{equation}\label{def:enlarged-wiener-space}
  (\tilde{\Omega},\tilde{\mathscr{F}},\tilde{\prob}) \coloneq \left(\mathbb{R},\mathscr{B}(\mathbb{R}),\gamma\right)^{\mathbb{N}}\times\left(\mathbb{R},\mathscr{B}(\mathbb{R}),\gamma\right)^{\mathbb{N}}.
\end{equation}

Since $\mathbb{N}^2$ is equipotent with $\mathbb{N}$, an enlarged Wiener space is actually a Wiener space.
In particular, we typically do not explicitly refer to this enlarging construction.

\subsubsection{The Wiener chaoses}

The \emph{Hermite polynomials}
\begin{equation*}
  H_{k}(x) \coloneq (-1)^{k} \mathrm{e}^{x^{2}/2} \frac{\mathrm{d}^{k}}{\mathrm{d} x^{k}} \mathrm{e}^{-x^{2}/2}, \qquad k \in \mathbb{N},
\end{equation*}
form a Hilbert basis of $L^{2}(\gamma)$.
  For $m \in \mathbb{N}$, the \emph{$m$-th Wiener chaos} $\mathcal{W}_{m}$ is defined as the $L^{2}(\prob)$-closure of the linear span of functions of the form $\prod_{i=0}^\infty H_{k_i}$ where the $k_{i}$'s satisfy $\sum_{i=0}^\infty k_i=m$.
The above product is finite since $H_0(x)=1$ and only finitely many integers $(k_i)_{i\ge 0}$ are non zero.
For $m = 0$, we find that $\mathcal{W}_{0}$ is the linear space of constant functions.
Importantly, Wiener chaoses provide the orthogonal decomposition
\begin{equation*}
  L^2(\prob)=\bigoplus_{m=0}^\infty \mathcal{W}_m.
\end{equation*}
We sometimes work in a finite sum of Wiener chaoses.
Accordingly, let us define
\begin{equation*}
  \mathcal{W}_{\leq m} \coloneq \bigoplus_{k=0}^m \mathcal{W}_k.
\end{equation*}

\subsubsection{Hypercontractivity \& equivalence of norms}\label{sec:hypercontractivity}
  We often use that on $\mathcal{W}_{\leq m}$ all the $L^{p}(\prob)$-norms are equivalent.
Namely, for all $m \in \mathbb{N}$, and $1 \leq p < q < \infty$, there exists $c = c_{m,p,q}$ such that
\begin{equation}\label{eq:hypercontractivity}
  \norm{F}_{p} \leq \norm{F}_{q} \leq c \norm{F}_{p}, \qquad F \in \mathcal{W}_{\leq m}.
\end{equation}

This fact is well-known in the range $1 < p < q < \infty$ as a consequence of \emph{hypercontractivity} estimates, for instance \cite[Cor.\ 2.8.14]{NourdinPeccati}.
The equivalence of norms can then be extended to the case $p = 1$ with an interpolation argument that we recall now.
Fix $p =1$ and $q \in (p,\infty)$.
Of course, we only need to show the last inequality in \cref{eq:hypercontractivity}.
Take $F \in \mathcal{W}_{\leq m}$.
A celebrated interpolation inequality, that is a consequence of Hölder's inequality, states that
\begin{equation*}
  \norm{F}_{p_{\theta}} \leq \norm{F}_{p_{0}}^{1-\theta} \norm{F}_{p_{1}}^{\theta}, \quad p_{0},\, p_{1} \in [1,\infty],\, \theta \in (0,1),
\end{equation*}
where $\frac{1}{p_{\theta}} \coloneq \frac{1-\theta}{p_{0}} + \frac{\theta}{p_{1}}$.
With $p_{0} = p = 1$, and $p_{1} = q+1$, there exists $\theta \in (0,1)$ such that $p_{\theta} = q$. 
Thus, we find, by hypercontractivity:
\begin{equation*}
  \norm{F}_{q} \leq \norm{F}_{1}^{1-\theta} \norm{F}_{q+1}^{\theta} \leq c_{m,q,q+1} \norm{F}_{1}^{1-\theta} \norm{F}_{q}^{\theta}.
\end{equation*}
This gives the announced extension.

\subsubsection{Reduction to finitely generated Wiener chaoses}\label{sec:finitely-generated}

For the sake of simplicity, we conveniently work with \textit{finitely generated} Wiener chaoses $\mathcal{W}_m^{(0)}$ defined as the linear span of functions of the form $\prod_{i=0}^\infty H_{k_i}$ with the $k_{i}$'s subject to $\sum_{i=0}^\infty k_i=m$.
This simplification avoids the use of infinite dimensional operators and allows us to manipulate instead matrices.
Although we state \cref{th:main-negative-moments,th:main-negative-moments-vector} for general Wiener chaoses, it is sufficient to establish them on $\mathcal{W}_m^{(0)}$ instead of $\mathcal{W}_m$.
Working in this finite setting, we show that for all $q\ge 1$ there exists $\delta_q,C_q>0$ such that
\begin{equation}\label{eq:main-bound-finitely-generated}
  d_{\text{FM}}(F,\mathcal{N}(0,1)) \leq \delta_q\Rightarrow \Esp*{\Gamma[F,F]^{-q}}\le C_q, \qquad F \in \mathcal{W}_{m}^{(0)}.
\end{equation}
  Let us show that by density of $\mathcal{W}_{m}^{(0)}$ in $\mathcal{W}_{m}$, \cref{eq:main-bound-finitely-generated} is actually sufficient to conclude on $\mathcal{W}_{m}$.
  Indeed, by continuity of $F \mapsto d_{\text{FM}}(F, \mathcal{N}(0,1))$ with respect to the $L^{2}(\prob)$-topology, we find that the condition on the left-hand side of \cref{eq:main-bound-finitely-generated} is $L^{2}(\prob)$-closed.
  Moreover, for $F_{n} \to F$ in $L^{2}(\prob)$, on $\mathcal{W}_{\leq m}$, $\Gamma[F_{n}, F_{n}] \to \Gamma[F,F]$ in $L^{2}$, and thus, up to extraction, convergence almost sure.
  By Fatou's lemma, we have
  \begin{equation*}
  \Esp*{ \Gamma[F,F]^{-q}} \leq \liminf_{n \to \infty} \Esp*{ \Gamma[F_{n}, F_{n}]^{-q}} \leq C_{q}.
\end{equation*}
We often work with polynomial random variables living finite sums of chaoses.
Thus, we define
\begin{equation*}
  \mathcal{W}_{\leq m}^{(0)} \coloneq \bigoplus_{k=0}^{m} \mathcal{W}_{k}^{(0)}.
\end{equation*}

\subsection{Malliavin calculus}\label{sec:malliavinop}

We introduce the operators coming from Malliavin's calculus used in this article.
We refer to the textbooks \cite{nualart2006malliavin,NourdinPeccati} for a broader introduction.
Since we work with polynomial mappings evaluated in finite-dimensional Gaussian vectors, we skip any technical considerations regarding domain and integrability, that frequently appear in Malliavin calculus.

\subsubsection {The square field operator}

We fix a Gaussian vector $\vec{N}=(N_1,\cdots,N_{K})$.
Given two multivariate polynomial mappings $F$ and $\tilde{F} \in \mathbb{R}[N_{1}, \dots, N_{K}]$, the \textit{square field operator} of $F$ and $\tilde{F}$ is

\begin{equation}\label{carrechamp}
  \Gamma\left[F,\tilde{F}\right] \coloneq \sum_{i=1}^{K} \frac{\partial F}{\partial N_i}(\vec{N}) \frac{\partial \tilde{F}}{\partial N_{i}}(\vec{N}).
\end{equation}

If $\vec{F}=(F_1,\ldots, F_d)$ is a random vector whose coordinates are as above, the Malliavin matrix $\Gamma(\vec{F})$ is the $d\times d$ positive symmetric random matrix defined by $\Gamma(\vec{F})_{i,j} \coloneq \Gamma[F_i,F_j]$. 

When considering $F$ and $\tilde{F} \in \mathbb{R}[N_1,\cdots,N_K,G_1,\cdots,G_{K'}]$ where $\vec{N}$ and $\vec{G}$ are independent standard Gaussian vectors, we set

\begin{align}
  \label{def:gamma-partial-N} \Gamma_N\left[F,F\right] & \coloneq \sum_{i=1}^K \frac{\partial F}{\partial N_i}(\vec{N},\vec{G}) \frac{\partial \tilde{F}}{\partial N_{i}}(\vec{N}, \vec{G}),\\
  \label{def:gamma-partial-G} \Gamma_G\left[F,F\right]  & \coloneq \sum_{i=1}^{K'} \frac{\partial F}{\partial G_i}(\vec{N},\vec{G}) \frac{\partial \tilde{F}}{\partial G_{i}}(\vec{N}, \vec{G}).
\end{align}
In this way, $\Gamma\left[F,F\right]=\Gamma_N\left[F,F\right]+\Gamma_G\left[F,F\right]$.
Similarly, when $\vec{F}=(F_1,\ldots, F_d)$ is a random vector, the conditional Malliavin matrices $\Gamma_N(\vec{F})$ and $\Gamma_G(\vec{F})$ are defined by $\Gamma_N(\vec{F})_{i,j} \coloneq \Gamma_N[F_i,F_j]$ and $\Gamma_G(\vec{F})_{i,j} \coloneq \Gamma_G[F_i,F_j]$.
We recall that $\Gamma(F, \tilde{F}) = \psh{\mathsf{D} F}{\mathsf{D} \tilde{F}}$.
By \cite[Prop.\ 1.2.2]{nualart2006malliavin}, $\mathsf{D}$ is continuous on $\mathcal{W}_{\leq m}$, and thus the operators $\Gamma$ thus defined can be extended by density to $\mathcal{W}_{\leq m}$ for all $m \in \mathbb{N}$.

\subsubsection{Malliavin's Lemma}\label{s:malliavin-lemma}
As a consequence of the seminal work of Malliavin \cite{Malliavin} concerning the proof of the H\"{o}rmander criterion, a random vector $\vec{F}$ of the Wiener space which is sufficiently smooth in some sense and such that $\det(\Gamma(\vec{F}))$ has negative moments at any order, has a smooth density.
Moreover, it is a quantitative statement enabling to bound uniform norms of the derivatives of the densities with respect to negative moments of the determinant of the Malliavin matrix.
In the framework of random vectors whose components are in a finite sum of Wiener chaoses, this result takes the following simpler form.

\begin{lemma}\label{th:malliavin-lemma}
  Let $m$ and $d \in \mathbb{N}^{*}$, and $q \in \mathbb{N}$.
  Then, there exist $q' \in \mathbb{N}$ and $C > 0$, both depending only on $(m,d,q)$ such that
  \begin{equation*}
    \|f\|_{W^{q,1}(\mathbb{R}^d)}\leq C~\Esp*{\det \Gamma(\vec{F})^{-q'}},
  \end{equation*}
  for any random vector $\vec{F} \in \mathcal{W}_{\leq m}^d$ with its Euclidean norm $\norm{\vec{F}}$ normalized such that $\Esp*{\norm{\vec{F}}^2}=d$, and whose corresponding density is denoted by $f$.
\end{lemma}

  \begin{proof}
    This lemma is nowadays rather standard in Malliavin calculus, and relies on successive integrations by parts.
    References \cite[Thm.\ 2.2]{HMPreg} or \cite[Thm.\ 3.2]{hairer2011malliavin} provide similar statements.
    The exact statement of the theorem comes from \cite[Prop.\ 2.1.4]{nualart2006malliavin} with the choice $G = 1$ and $u_{j} = \mathsf{D} F_{j}$.
    See the paragraph after the proof, in particular \cite[Eq.\ (2.32)]{nualart2006malliavin} and the subsequent equation.
  \end{proof}

\subsubsection{Normal convergence and carré du champ on Wiener chaoses}\label{s:malliavin-stein}

We repeatedly call upon the following emblematic result from the literature of limit theorems for Wiener chaoses: a sequence of chaotic random variables is asymptotically normal if and only if its carré du champ converges to a constant.
We state the most general version for vectors.

\begin{theorem}[{\cite{nualart2008central,nourdin2009stein,nourdin2010multivariate}}]\label{th:fourth-moment-stein}
  Let $d \in \mathbb{N}^{*}$, and $m_1,\ldots,m_d \in \mathbb{N}^{*}$. 
  Then,
  \begin{equation*}
    \vec{F}_{n} \xrightarrow[n \to \infty]{law} \mathcal{N}(0, I_{d}) \Longleftrightarrow \Gamma(\vec{F}_{n}) \xrightarrow[n \to \infty]{L^{2}(\prob)} \mathrm{diag}(m_{1}, \dots, m_{d}), \qquad (\vec{F}_{n}) \subset \prod_{i} \mathcal{W}_{m_{i}}.
  \end{equation*}
  Moreover, with $C > 0$ a constant depending only on $m_1,\ldots,m_d$,
  \begin{equation*}
    d_{\mathrm{FM}}(\vec{F},\mathcal{N}(0,I_d)) \leq C \norm{\Gamma(\vec{F})-\mathrm{diag}(m_1,m_2,\cdots,m_d)}_{L^2}, \qquad \vec{F} \in \prod_{i} \mathcal{W}_{m_{i}}.
  \end{equation*}
\end{theorem}

\section{The second Wiener chaos}\label{secondchaos}
In this section, we study the case of the second Wiener chaos $\mathcal{W}_2$.
Since to every element $F \in \mathcal{W}_{2}$ corresponds a symmetric quadratic form, reduction theory makes the analysis much easier.
We present here some of these tools, and we use them to prove \cref{th:main-negative-moments} in this simpler context.
The proof in the general case relies in a crucial way on the result for $\mathcal{W}_{2}$.

\subsection{Quadratic forms}

\subsubsection{Diagonalization}
Every element $F \in \mathcal{W}^{(0)}_2$ is of the form
\begin{equation*}
  F = \sum_{i,j=1}^{K} a_{i,j} X_{i,j}, \qquad X_{i,j} \coloneq
  \begin{cases}
     N_i N_j, & \text{if}\ i \ne j,
\\   N_i^2-1, & \text{if}\ i=j,
  \end{cases}
\end{equation*}
for some $K \in \mathbb{N}^{*}$, and where $A = (a_{i,j})$ is a symmetric matrix of size $K \times K$.
Thus, defining
\begin{equation*}
  q(x,x) \coloneq \sum_{i,j = 1}^{K} a_{i,j} x_{i} x_{j}, \quad x \in \mathbb{R}^{K},
\end{equation*}
a diagonalization procedure yields
\begin{equation*}
  F = q(\vec{N}, \vec{N}) - \Esp*{ q(\vec{N}, \vec{N}) } = \sum_{i=1}^{K} \lambda_{i} (\tilde{N}_{i}^{2} - 1),
\end{equation*}
where $(\lambda_{i})$ is the spectrum of $A$, and $(\tilde{N}_i)$ is a new sequence of independent standard Gaussian variables, obtained by an orthogonal transformation of $(N_k)$.

\subsubsection{Spectral considerations}\label{sec:spectral}
The following spectral quantities play a crucial role in our proofs.
Given a symmetric matrix $A$ with spectrum $\{\lambda_i\}$, we define its \emph{spectral remainders} 
\begin{equation}\label{def:spectral-remainders}
  \mathcal{R}_q(A) \coloneq \sum_{i_1\not=i_2\not=\cdots\not=i_q} \lambda_{i_1}^2\cdots\lambda_{i_q}^2, \qquad q \in \mathbb{N}^{*}.
\end{equation}
  In all the paper, the above notation $i_{1} \ne i_{2} \ne \dots \ne i_{q}$ indicates summation over all pairwise distinct indices.

We conveniently generalize the definition of $\mathcal{R}_q$ to non square matrices.
In this case, we replace the spectrum by the singular values of $A$.
Namely,
\begin{equation}\label{def:spectral-remainders-rectangular}
  \mathcal{R}_q(A) \coloneq \sum_{i_1\not=i_2\not=\cdots\not=i_q} \mu_{i_1}\cdots\mu_{i_q}, \qquad q \in \mathbb{N}^{*},
\end{equation}
where $\{\mu_i\}$ is the spectrum of ${}^tAA$.
We highlight that $\mathcal{R}_q(A)=\mathcal{R}_q({}^t A)$.
The generalized Cauchy--Binet formula \cite{CauchyBinet} expresses $\mathcal{R}_q(A)$ in terms of extracted determinants:
\begin{equation}\label{eq:cauchy-binet}
  \mathcal{R}_q(A)=\sum_{|I|=|J|=q} \det( A_{I,J})^2,
\end{equation}
where $A_{I,J}$ is the extracted matrix $A_{I,J} \coloneq (a_{i,j})_{i\in I, j\in J}$.

If $A$ is symmetric with spectrum ordered by decreasing absolute values $|\lambda_1|\ge|\lambda_2|\ge \cdots$, we define the \emph{squared distance to matrices of rank less $q$}
\begin{equation*}
  r_q(A) \coloneq \Inf*{ \norm{A-B}^2 : \rk(B)\le q-1 } = \sum_{i\ge q} \lambda_i^2, \qquad q \in \mathbb{N}^{*},
\end{equation*}
  with $\norm{\cdot}$ designating the Euclidean norm on matrices.
\begin{proof}[Proof of the equality]
This is a consequence of the Eckart--Young--Mirsky Theorem for the Frobenius norm \cite[7.4.15]{horn2012matrix}.
This theorem states that the orthogonal projection of $A$ on the matrices of rank less or equal than $q$ is
\begin{equation*}
  \mathrm{proj}_{\leq q}(A) = P D_{q} {}^{t}P,
\end{equation*}
where $D_{q} \coloneq \mathrm{diag}(\lambda_{1}, \dots, \lambda_{q}, 0 \dots, 0)$, with $(\lambda_{i})$ the spectrum of $A$, and $P$ is the orthogonal matrix that diagonalizes $A$.
\end{proof}
Since $\mathcal{R}_q(A) = 0$ if and only if $A$ has rank less than $q$, in some sense, $\mathcal{R}_q(A)$ also measures the distance of $A$ to matrices with rank $q-1$ or less.
Actually, the two quantities $\mathcal{R}_q(A)$ and $r_q(A)$ are comparable.

\begin{lemma}
  Let $q \in \mathbb{N}$ and $A$ be a symmetric matrix,
  \begin{align}
  & \mathcal{R}_{q-1}(A) r_q(A)\le \mathcal{R}_q(A)\le q\mathcal{R}_{q-1}(A) r_q(A); \label{ineq:compar1} 
\\& \prod_{i=1}^q r_i(A)\le \mathcal{R}_q(A)\le q!\prod_{i=1}^q r_i(A); \label{ineq:compar2} 
\\& r_q(A)^q\le \mathcal{R}_q(A)\le q!r_1(A)^{q-1} r_q(A). \label{ineq:compar3} 
  \end{align}
\end{lemma}

\begin{proof}
  We only prove \eqref{ineq:compar1}, since \cref{ineq:compar2} proceeds from an immediate induction, and \cref{ineq:compar3} follows from the monotony of $(r_q(A))_q$.
  For the second inequality in \cref{ineq:compar1}, we write
\begin{equation*}
  \begin{split}
    \mathcal{R}_q(A)&= q! \sum_{i_1<i_2<\cdots<i_{q}} \lambda_{i_1}^2\cdots\lambda_{i_{q}}^2
                  \\&= q! \sum_{i_1<i_2<\cdots<i_{q-1}} \lambda_{i_1}^2\cdots\lambda_{i_{q-1}}^2\underbrace{\sum_{i_q>i_{q-1}}\lambda_{i_q}^2}_{\le r_q(A)}
                  \\&\le q!\left(\sum_{i_1<i_2<\cdots<i_{q-1}} \lambda_{i_1}^2\cdots\lambda_{i_{q-1}}^2\right)r_q(A) = q \mathcal{R}_{q-1}(A) r_q(A).
  \end{split}
\end{equation*}
For the first inequality, we have
\begin{equation*}
\mathcal{R}_q(A)=\sum_{i_1\not=\cdots \not=i_{q-1}}\lambda_{i_1}^2\cdots\lambda_{i_{q-1}}^2\underbrace{\sum_{i_q\notin\{i_1,\cdots,i_{q-1}\}}\lambda_{i_q}^2}_{\ge r_q(A)}\ge \mathcal{R}_{q-1}(A)r_q(A).
\end{equation*}
This completes the proof.
\end{proof}

Define the \emph{spectral radius} of $A$, $\rho(A) \coloneq \max_{\lambda \in \mathrm{spec}(A)} \abs{\lambda}$.
From the above estimates, we deduce the following useful estimate.
\begin{lemma}\label{th:spectral-remainder-radius}
Let $A$ be a $n \times n$ symmetric matrix such that $Tr(A^2)=1$ and $\rho(A)\leq \frac{1}{q}$, then
\begin{equation*}
  \mathcal{R}_q(A)\ge \prod_{k=1}^{q-1}(1-k\rho(A)).
\end{equation*}
\end{lemma}

\begin{proof}
Writing as before $\lambda_1,\ldots,\lambda_n$ for the eigenvalues of $A$, we have $\sum_{i=1}^n \lambda_i^2=1$ and $\sup_{i} |\lambda_i|=\rho(A)$, so $r_i(A)\ge 1-(i-1)\rho(A)$ and the result follows from \cref{ineq:compar2}.
\end{proof}

\subsection{Proof of the main theorem on the second Wiener chaos}\label{s:proof-second-chaos}
With the spectral tools introduced above, we now establish \cref{th:main-negative-moments} for elements of $\mathcal{W}_{2}$.
\cref{th:main-negative-moments} follows from the following more general estimate.
\begin{proposition}\label{th:main-degree-2}
  Let $q \in \mathbb{N}$.
  There exists $C > 0$ such that 
  \begin{equation*}
    \Esp*{\Gamma[F,F]^{-q}}\leq C\mathcal{R}_{2q+1}(A)^{-\frac{1}{2}}, \qquad F\in \mathcal{W}_2^{(0)}.
  \end{equation*}
\end{proposition}

\begin{proof}
  We assume can that $q \ne 0$, otherwise the claim is trivial.
Let $F \in \mathcal{W}_2^{(0)}$.
Consider the matrix $A$ associated to $F$ through the quadratic form.
$A$ is of size $K \times K$ for some $K \in \mathbb{N}$, and has eigenvalues $(\lambda_k)_{1\le k \le K}$.
Diagonalizing $A$, we can assume that
\begin{equation*}
  F = \sum_{k=1}^{K} \lambda_{k}(N_k^2-1).
\end{equation*}
Fix $t \in \mathbb{R}$.
It follows that
\begin{align*}
   & \Gamma[F,F] = \sum_{k=1}^{K} 4\lambda_{k}^2N_k^2;
 \\& \Esp*{\mathrm{e}^{-\frac{t^2}{2}\Gamma[F,F] }} = \prod_{i=1}^{K} \frac{1}{\left(1+4t^2\lambda_{i}^2\right)^{1/2}}.
\end{align*}
Expanding the product gives the trivial bound $\prod_{i=1}^K(1+t^2\lambda_i^2)\ge t^{2q}\mathcal{R}_q(A)$.
Thus, we get
\begin{equation*}
  \Esp*{\mathrm{e}^{-\frac{t^2}{2}\Gamma[F,F] }} \le \frac{1}{t^{q} \cdot \mathcal{R}_q(A)^{1/2}}.
\end{equation*}
Using that
\begin{equation}\label{eq:x-minus-q-integral}
  x^{-q} = c_{q} \int_{0}^{\infty} t^{q-1} \mathrm{e}^{-tx} \mathrm{d} t, \qquad x > 0, \, c_{q} \coloneq \frac{1}{(q-1)!},
\end{equation}
we find that,
\begin{equation*}
  \Esp*{\Gamma[F,F]^{-q}} = c_q \int_{\mathbb{R}}|t^{2q-1}| \Esp*{ \mathrm{e}^{-\frac{t^2}{2}\Gamma[F,F] } } dt \leq \frac{c_q}{\mathcal{R}_{2q+1}(A)^{1/2}}.
\end{equation*}
The result follows.
\end{proof}

We now complete the proof in the case of the second Wiener chaos.
\begin{proposition}
\cref{th:main-negative-moments} holds for $m=2$.
In particular, a sequence $(F_n) \subset \mathcal{W}_{2}$ converging to a Gaussian distribution satisfies: for every $q \in \mathbb{N}$ there exist $N \in \mathbb{N}$ and $C > 0$ such that
\begin{equation*}
 \Esp*{ \Gamma\left[F_n,F_n\right]^{-q} }\le C, \qquad n \geq N.
\end{equation*}
\end{proposition}

\begin{proof}
  We assume can that $q \ne 0$, otherwise the claim is trivial.
  As explained in \cref{sec:finitely-generated}, it is sufficient to prove our claims on $\mathcal{W}_{2}^{(0)}$.
  By density, without loss of generality we assume that $(F_n) \subset \mathcal{W}_2^{(0)}$.
We denote by $A_n$ the associated matrix of size $k_{n} \times k_{n}$, and $(\lambda_{i,n})$ its spectrum.
By \cref{th:main-degree-2}, it is sufficient to bound from below the quantities $\mathcal{R}_q(A_n)$.
By assumption, we have
\begin{equation*}
  \begin{cases}
    & \Esp*{F_n^2} \xrightarrow[n \to \infty]{} 1,
  \\& \Esp*{F_n^4} \xrightarrow[n \to \infty]{} 3;
  \end{cases}
  \quad \Longleftrightarrow \quad
  \begin{cases}
    & \sum_{i=1}^{k_n}\lambda_{i,n}^2 \xrightarrow[n \to \infty]{} \frac{1}{2},
  \\&  \sum_{i=1}^{k_n}\lambda_{i,n}^4 \xrightarrow[n \to \infty]{} 0.
  \end{cases}
\end{equation*}
This implies that $\rho(A_{n}) \to 0$.
Since, by \cref{th:spectral-remainder-radius}, $\mathcal{R}_q(A_n)\ge \prod_{k=1}^{q-1}(1-k\rho(A_n))$, we deduce that $\mathcal{R}_q(A_n)$ is bounded by below for $n$ large, and we conclude. 
\end{proof}

\section{Sharp operator}\label{s:sharp}
In this section, we establish estimates regarding Malliavin derivatives, when we specifically choose to take derivatives in the directions of a Gaussian space.
In this case, the Malliavin derivative is an element of an enlarged Wiener space.
Thus, we can use Gaussian analysis to conclude.
Through these estimates, we obtain that negative moments of $\Gamma[F,F]$ are estimated by the spectral remainders of the Malliavin Hessian of $\mathsf{D}^{2}F$.
The estimates of this section are akin to our results from \cite{HMPreg} obtained in a more general setting.
For the sake of completeness, we present self-contained arguments tailor-made to the case of normal convergence on Wiener chaoses.
For simplicity, we state our results in finite sums of Wiener chaos, they still hold for more general random variables by density arguments, similar to that of \cref{sec:finitely-generated}.

\subsection{Iterated sharp operators}
The \emph{sharp operator}, introduced by Bouleau \cite{bouleau2003error} with a slightly different definition, is a convenient way to interpret the Malliavin derivative.
For a standard Gaussian vector $\vec{N}=(N_1,\cdots,N_{K})$ and  a polynomial mapping $F\in\mathbb{R}\left[N_1,N_2,\cdots,N_{K}\right]$ the \emph{Bouleau derivative} of $F$ is
\begin{equation}\label{sharp}
  \sharp[F] \coloneq \sum_{i=1}^{K} \frac{\partial F}{\partial N_i}(\vec{N}) G_i,
\end{equation}
where $\vec{G}=(G_1,\cdots,G_{K})$ is a Gaussian vector independent of $\vec{N}$.
This operator intimately relates to the square field operator through the Laplace-Fourier identity
 \begin{equation*}
   \Esp*{\mathrm{e}^{\mathrm{i} t \sharp\left[F\right]}}=\Esp*{\exp\left(-\frac{t^2}{2}\Gamma\left[F,F\right]\right)}.
 \end{equation*}
 Our work exploits other connections between $F$, $\Gamma[F,F]$, and $\sharp[F]$, and they will become apparent to the reader in the rest of this work.
 Intuitively, the random variable $\sharp[F]$ is simpler than $F$, in view of the independence between the terms $G_i$ and $\frac{\partial F}{\partial N_i}(\vec{N})$.

We generalize the definition to cover iterated Malliavin derivatives.
We fix $(G_{i,j})$ a sequence of independent standard Gaussian variables, independent of $\vec{N}$.
For a polynomial function $F \in \mathbb{R}[N_1,\ldots,N_{K}]$, we let
\begin{equation*}
  \sharp^k[F] \coloneq \sum_{1\le i_1,\ldots,i_k\leq K}\frac{\partial^k F}{\partial N_{i_1}\cdots \partial N_{i_k}}(\vec{N}) G_{1,i_1}\cdots G_{k,i_k}, \qquad k = 1,2, \dots.
\end{equation*}
  When $k=1$, the definition of $\sharp^{k}$ is consistent with that of $\sharp$.
  When $k = 0$, the above formula is understood as $\sharp^{0}[F] = F$.
  By density, the operators $\sharp^{k}$ extends to $\mathcal{W}_{m}$ (to see this, observe that $\sharp^{k}$ is simply the iterated Malliavin derivative when we choose $\mathfrak{H}$ to be a Gaussian space, and conclude by \cite[Prop.\ 2.3.4]{NourdinPeccati}).
We regard $\sharp^k[F]$ as an element of an enlarged Wiener space, in the sense of \cref{def:enlarged-wiener-space}, generated by the variables $(N_k)$ and $(G_{i,j})$.

These operators satisfy the following elementary properties.
\begin{enumerate}
\item $F\in \mathcal{W}_m\Rightarrow \sharp^k[F]\in \mathcal{W}_m$ and $F\in \mathcal{W}_{\leq m}\Rightarrow \sharp^k[F]\in \mathcal{W}_{\leq m}$,
\item $F\in \mathcal{\mathcal{W}}_m\Rightarrow \Var*{\sharp^k[F]}=m(m-1)\cdots (m-k+1)) \Var*{F}$.
\end{enumerate}
The first point follows immediately, since for $F\in \mathcal{W}_{m}$, the partial derivatives with respect to the $N_{i}$'s are in $\mathcal{W}_{m-1}$, and in view of the independence of $\vec{N}$ and $\vec{G}$.
For the second point, we use a consequence of the integration by part formula, which gives for $F\in\mathcal{W}_m$, see for instance \cite[Prop.\ 1.2.2]{nualart2006malliavin}, that
\begin{equation*}
  \Esp*{ F^2 } = \frac{1}{m} \sum_{i=1}^{K} \Esp*{\left(\frac{\partial F}{\partial N_i}\right)^2} = \frac{1}{m} \Esp*{ \sharp[F]^{2} }.
\end{equation*}
Iterating this formula gives the second point.

\subsection{Negative moments estimates}\label{s:negative-moments-estimates}
Estimates on the negative moments of $\Gamma\left[\sharp^k[F],\sharp^k[F]\right]$ yield estimates on those of $\Gamma[F,F]$.
\begin{proposition}\label{supersharpnegmoments}
  Let $m$ and $k \in \mathbb{N}^{*}$ with $k\le m$.
  For all $q \in \mathbb{N}$, there exists $q' \in \mathbb{N}$ and $C > 0$ such that for every $F\in\mathcal{W}_{\leq m}$ with $\Esp*{F^{2}} =1$,
\begin{equation*}
  \Esp*{\Gamma\left[F,F\right]^{-q}}\le C \Esp*{\Gamma\left[\sharp^k[F],\sharp^k[F]\right]^{-q'}}.
\end{equation*}
\end{proposition}
\begin{proof}
Let $F\in\mathcal{W}_{\leq m}$.
Write $F_{k}=\sharp^k[F]$ for $k=0,\ldots, m$.
For $k\ge 1$, the following induction relation holds
\begin{equation*}
  F_k=\sum_{i=1}^{K} \frac{\partial F_{k-1}}{\partial N_i} G_{i,k}.
\end{equation*}
This shows that $F_k$ has the same law as $V_k^{\frac{1}{2}} N$, where
\begin{equation*}
  V_k = \sum_{i=1}^{K} \left(\frac{\partial F_{k-1}}{\partial N_i}\right)^2=\Gamma_N[F_{k-1},F_{k-1}],
\end{equation*}
 and $N\sim\mathcal{N}(0,1)$ is independent of $V_k$.
 In particular, we have the Fourier-Laplace identity
 \begin{equation}\label{eq:fourier-laplace}
   \Esp*{\mathrm{e}^{\mathrm{i} t F_k}} = \Esp*{ \mathrm{e}^{-\frac{t^2}{2}\Gamma_N\left[F_{k-1},F_{k-1}\right]} }.
 \end{equation} 
 Fix $m$, $q$, and $k$ as in the theorem.
 In Malliavin's result (\cref{th:malliavin-lemma}) take the $q' \in \mathbb{N}$ associated with $2q+1$.
 If $\Esp*{ \Gamma[ F_{k}, F_{k}]^{-q'} } = \infty$, then the statement is empty and the proof is complete.
 Thus, we assume that $\Esp*{ \Gamma[F_{k},F_{k}]^{-q'} } < \infty$.
 By \cref{th:malliavin-lemma}, the density $f_k$ of $F_k$ belongs to $W^{2q+1,1}(\mathbb{R})$, and 
\begin{equation*}
  \norm{f_k}_{W^{2q+1,1}(\mathbb{R})}\leq C\Esp*{ \Gamma[F_{k},F_{k}]^{-q'} }.
\end{equation*}
\cref{eq:sobolev-embeddings-fourier} gives the bound on the Fourier transform:
\begin{equation*}
  \norm{f_{k}}_{N^{j}} |\leq  \norm{f_k}_{W^{2q+1,1}(\mathbb{R})}, \qquad j\leq 2q+1.
\end{equation*}
Therefore, up to changing the constant $C$,
\begin{equation*}
  \left|t^{2 q-1} \Esp*{\mathrm{e}^{\mathrm{i} t F_k}}\right|\le \frac{C}{t^2+1}\Esp*{\Gamma[F_k,F_k]^{-q'}}, \qquad t \in \mathbb{R}.
\end{equation*}
Reporting in the Fourier-Laplace identity \cref{eq:fourier-laplace} yields
\begin{equation*}
  |t^{2 q-1}| \Esp*{ \exp\left(-\frac{t^2}{2}\Gamma_N\left[F_{k-1},F_{k-1}\right]\right) }\leq  \frac{C}{t^2+1}\Esp*{ \Gamma[F_k,F_k]^{-q'} }, \qquad t \in \mathbb{R}.
\end{equation*}
Using \cref{eq:x-minus-q-integral}, we find that
\begin{equation*}
  \Esp*{ \Gamma_N\left[F_{k-1},F_{k-1}\right]^{-q} } = c_{q} \int_{0}^{\infty}\abs{t}^{2q-1} \Esp*{ \mathrm{e}^{-\frac{t^{2}}{2} \Gamma_{N}[ F_{k}, F_{k} ] } } \mathrm{d} t \leq C\Esp*{ \Gamma[F_k,F_k]^{-q'} },
\end{equation*}
where $\Gamma_{N}$ is defined in \cref{def:gamma-partial-N}.
Since
\begin{equation*}
  \Gamma[F_{k-1},F_{k-1}]=\Gamma_N[F_{k-1},F_{k-1}]+\Gamma_G[F_{k-1},F_{k-1}]\ge \Gamma_N[F_{k-1},F_{k-1}],
\end{equation*}
we deduce that
\begin{equation*}
  \Esp*{ \Gamma\left[F_{k-1},F_{k-1}\right]^{-q} }\leq C\Esp*{ \Gamma[F_k,F_k]^{-q'} }.
\end{equation*}
The statement follows by an immediate induction.
\end{proof}

Combining \cref{supersharpnegmoments} for $k=2$ with \cref{th:main-degree-2}, we obtain the estimate for negative moments of $\Gamma[F,F]$ in terms of spectral remainders for $\nabla^{2}F$, the Hessian matrix of $F$, that is
\begin{equation*}
  \nabla^{2}F \coloneq \left(\frac{\partial^2F}{\partial N_i\partial N_j}\right)_{1 \leq i,j \leq K}, \qquad F \in \mathbb{R}[N_{1}, \dots, N_{K}].
\end{equation*}

\begin{proposition}\label{th:bound-gamma-spectral-hessian}
  Let $m \in \mathcal{N}_{>0}$ and $q \in \mathbb{N}$, there exist $q' \in \mathbb{N}$ and $C > 0$ such that for every $F\in \mathbb{R}[N_1,\ldots,N_{K}]$ of degree $m$, with $\Var*{F} =1$, 
\begin{equation*}
  \Esp*{ \Gamma[F,F]^{-q} }\leq C\Esp*{ \mathcal{R}_{q'}\left(\nabla^{2}F\right)^{-\frac{1}{2}} }.
\end{equation*}
\end{proposition}

\begin{proof}
  Fix $m \in \mathbb{N}^{*}$ and $q \in \mathbb{N}$.
Let us denote by $(A_{i,j})$ the Hessian matrix of $F$, and let $\widetilde{F} \coloneq \sharp^2[F]$.
Then,
\begin{equation*}
  \widetilde{F}=\sum_{i,j\ge 1} A_{i,j}G_{i,1}G_{j,2}.
\end{equation*}
  By \cref{supersharpnegmoments}, there exist $q'' \in \mathbb{N}$ and $C > 0$ such that
\begin{equation*}
  \Esp*{ \Gamma[F,F]^{-q} }\leq C\Esp*{ \Gamma[\widetilde{F},\widetilde{F}]^{-q''} }.
\end{equation*}
Thus, it is sufficient to bound from above the right-hand side.
Since $A$ and $\vec{G}$ are independent, fixing a realization of the entries $A_{i,j}$, $\widetilde{F}$ is a function of the variables $G_{i,1}$ and $G_{i,2}$, and can be seen as an element of the second Wiener chaos, with associated Hessian matrix
\begin{equation*}
\widetilde{A}=
\begin{pmatrix}
0 & A\\
A & 0
\end{pmatrix}.
\end{equation*}
The characteristic polynomial of $\widetilde{A}$ is $t \mapsto \chi_{A}(t)\chi_{A}(-t)$ where $\chi_{A}$ stands for the characteristic polynomial of $A$.
Hence,
\begin{equation*}
  \mathrm{spec}(\widetilde{A}) = \set*{ \lambda, - \lambda : \lambda \in \mathrm{spec}(A) }.
\end{equation*}
This yields that $\mathcal{R}_{p}(A) \leq \mathcal{R}_{p}(\widetilde{A})$ ($p \in \mathbb{N}^{*}$).
Applying \cref{th:main-degree-2} gives $q' \coloneq 2q''+1$ and $C > 0$ such that
\begin{equation*}
    \esp_G[\Gamma_G[\widetilde{F},\widetilde{F}]^{-q''}] \leq C\mathcal{R}_{q'}(\widetilde{A})^{-\frac{1}{2}} \leq C\mathcal{R}_{q'}(A)^{-\frac{1}{2}},
\end{equation*}
where $\esp_G$ (resp. $\Gamma_G$) means that we only integrate (resp. derivate) with respect to the variables $G_{i,j}$. Using that $\Gamma[\widetilde{F},\widetilde{F}]\ge\Gamma_G[\widetilde{F},\widetilde{F}]$,and integrating with respect to $N$ we get 
\begin{equation*}
  \Esp*{ \Gamma[\widetilde{F},\widetilde{F}]^{-q''} }\leq C\Esp*{ \mathcal{R}_{q'}(A)^{-\frac{1}{2}} }.
\end{equation*}
This concludes the proof.
\end{proof}

\section{Proof of the main theorems}\label{s:proofs}
\subsection{Setup}
In this section we prove \cref{th:main-negative-moments}.
We proceed by induction on the degree $m$ of the chaos.
Let us define the property to be established.
\begin{equation}\tag{$\mathcal{P}(m)$}
  \begin{split}
    &\text{For every sequence}\ (F_n)_{n\ge 1} \subset \mathcal{W}_m,
  \\& \bracket*{F_n\xrightarrow[n\to\infty]{\text{Law}}~\mathcal{N}(0,1)} \Rightarrow
  \bracket*{\limsup_{n\to +\infty} \Esp*{\Gamma[F_n,F_n]^{-q}}<+\infty, \qquad q \in \mathbb{N}}.
\end{split}
\end{equation}
This property is equivalent to the following non sequential version.
\begin{equation*}\tag{$\mathcal{P}(m)$}
  \begin{split}
    & \forall q \in \mathbb{N},\ \exists \delta = \delta_{q} > 0,\, \exists C = C_{q} > 0 :
  \\& \forall F \in \mathcal{W}_{m},\, \bracket*{d_{\mathrm{FM}}(F,\mathcal{N}(0,1))\le \delta} \Rightarrow \bracket*{\Esp*{\Gamma[F,F]^{-q}}\le C}.
  \end{split}
\end{equation*}
\cref{secondchaos} establishes $\mathcal{P}(2)$.
Let us prove that for every $m \ge 3$, $\mathcal{P}(m-1) \Rightarrow \mathcal{P}(m)$.

We often use that controls on negative moments in $\mathcal{P}(m)$, are expressible in terms of small ball estimates.
More precisely, for every sequence of random variables $(X_n)$, we recall the following elementary equivalence:
\begin{enumerate}[wide]
  \item For every $q\ge 0$, there exists $N\in \mathbb{N}$ such that $\sup_{n\ge N} \Esp*{ |X_n|^{-q} }<+\infty$.
  \item For every $q\ge 0$, there exist $N\in \mathbb{N}$ and $C>0$ such that
    \begin{equation*}
      \forall\epsilon>0, \sup_{n \geq N} \Prob*{ |X_n|\le \epsilon }\le C\epsilon^q.
    \end{equation*}
\end{enumerate}

\subsection{The discretization procedure}
Through a discretization procedure, we obtain that $\mathcal{P}(m)$ is equivalent to the following vectorial version.
For $d \in \mathbb{N}^{*}$, we consider
\begin{equation*}\tag{$\mathcal{P}_d(m)$}
  \begin{split}
  & \text{For every sequence}\ (\vec{F}_n)_{n\ge 1} \subset \mathcal{W}_{\leq m}^d,
\\& \bracket*{\vec{F}_n\xrightarrow[n\to\infty]{\text{Law}}~\mathcal{N}(0,I_d)} \Rightarrow
    \bracket*{\limsup_{n\to +\infty}\Esp*{\det\Gamma (\vec{F}_n)^{-q} }<+\infty, \qquad q \in \mathbb{N}}.
  \end{split}
\end{equation*}
As above, it is equivalent to the non sequential version.
\begin{equation*}\tag{$\mathcal{P}_d(m)$}
  \begin{split}
    & \forall q \in \mathbb{N},\, \exists \delta=\delta_q>0,\, \exists C=C_q>0 : 
  \\& \forall \vec{F} \in \mathcal{W}_m^{d},\,
\bracket*{d_{\mathrm{FM}}(\vec{F},\mathcal{N}(0,I_{d}))\le \delta} \Rightarrow \bracket*{\Esp*{\det(\Gamma(\vec{F}))^{-q}}\le C}.
  \end{split}
\end{equation*}
In this section, we prove the implication $\mathcal{P}(m)\Rightarrow \bracket*{\forall d \in \mathbb{N}^{*},\, \mathcal{P}_d(m)}$ via a more general statement.

\begin{proposition}\label{prop:dimupgrade}
  Let $d$ and $m \in \mathbb{N}^{*}$.
  Consider a sequence $(\vec{F}_n) \subset \mathcal{W}_{\leq m}^{d}$ that is also $L^2$-bounded sequence.
  Then there is equivalence between the two following properties:
  \begin{enumerate}[(i),wide]
    \item\label{prop:dimupgrade:sphere} for every sequence $(\vec{a}_n)$ in the sphere $\mathbb{S}^{d-1}$, and all $q > 0$
    \begin{equation*}
      \limsup_{n\to +\infty} \Esp*{ \Gamma[\vec{F}_n\cdot \vec{a_n},\vec{F}_n\cdot \vec{a_n}]^{-q} }<+\infty.
    \end{equation*}
    \item\label{prop:dimupgrade:determinant} for all $q > 0$, $\displaystyle \limsup_{n\to +\infty} \Esp*{ \det\Gamma (\vec{F}_n)^{-q}  }<+\infty.$
\end{enumerate}
\end{proposition}

\begin{corollary}\label{cor:dimupgrade}
  For any $m\ge 2$, if $\mathcal{P}(m)$ holds then $\mathcal{P}_d(m)$ also holds for every $d \in \mathbb{N}^{*}$.
\end{corollary}

The proof of the proposition relies on a discretization procedure of the sphere.
Such procedure is frequently used in Malliavin calculus, for instance \cite[Lem.\ 4.7]{hairer2011malliavin}.
We use the following discretization result for the $d-1$-dimensional Euclidean sphere $\mathbb{S}^{d-1}$.

\begin{lemma}\label{discretization}
  For all $d \in \mathbb{N} \setminus \{1\}$ and $N \in \mathbb{N}$, there exist $C_d>0$ (not depending on $N$) and $\mathbb{S}^{d-1,N}\subset \mathbb{S}^{d-1}$ such that $\mathrm{Card}\left(\mathbb{S}^{d-1,N}\right)\le C_d N^d$ and
\begin{equation*}
  \forall a \in \mathbb{S}^{d-1},\,\exists b\in\mathbb{S}^{d-1,N},\,\text{such that}\,\|a-b\|\le \frac{C_d}{N}.
\end{equation*}
\end{lemma}
\begin{proof}
  We fix an positive integer $N$.
  Write $\lBrack-N, N \rBrack \coloneq \{-N, -N+1, \dots, N-1, N\}$.
  For every $I=(i_1,i_2,\cdots,i_d)\in\lBrack-N,N\rBrack^d$, we set $b_I \coloneq \frac{I}{N}$.
For every $a\in\mathbb{S}^{d-1}$, we may find $I\in\lBrack-N,N\rBrack^d$ such that all the coordinates of $b_I-a$ are $\le \frac{1}{N}$.
Hence,
\begin{equation*}
  |\|b_I\|-1|\le \|b_I-a\| \le \frac{\sqrt{d}}{N}.
\end{equation*}
Whenever $N\ge 2\sqrt{d}$, the above choice of $b_{I}$ yields $\|b_I\|\ge \frac{1}{2}$.
Thus, setting $a_I \coloneq \frac{b_I}{\|b_I\|}$, we get
\begin{equation*}
  \|a_I-a\| \le 2\left\|b_I-a\|b_I\|\right\|\le 2 \|b_I-a\|+2|1-\|b_I\||\le \frac{4\sqrt{d}}{N}. 
\end{equation*}
We define
\begin{equation*}
  \mathbb{S}^{d-1,N} \coloneq \set*{ \frac{b_I}{\|b_I\|}=a_I : I\in\lBrack-N,N\rBrack^d/\{0,\cdots,0\} }.
\end{equation*}
    We have proved that, provided that $N\ge 2\sqrt{d}$, for any $a\in\mathbb{S}^{d-1}$ we may find $a_I\in\mathbb{S}^{d-1,N}$ such that $\|a-a_I\|\le \frac{4\sqrt{d}}{N}$.
Besides $\text{Card}\left(\mathbb{S}^{d-1,N}\right)\le (2N+1)^d$.
\end{proof}

Let us prove the proposition.
\begin{proof}[Proof of {\cref{prop:dimupgrade}}]
    In all the proof, we write indistinctly $\norm{\cdot}$ for the Euclidean norm of a vector, or the Euclidean norm of a matrix, also known as its Hilbert--Schmidt norm.
  We first prove $\cref{prop:dimupgrade:sphere} \Rightarrow \cref{prop:dimupgrade:determinant}$, that is the only implication used in the paper.
  Assuming \cref{prop:dimupgrade:sphere}, we want to obtain a bound 
  \begin{equation*}
    \Prob*{ \det \Gamma(\vec{F}_n)\le \epsilon }\le C_q\epsilon^q, \qquad \varepsilon > 0,\, q\ge 0.
  \end{equation*}
    For any $S$ symmetric positive matrix $d\times d$ we have
\begin{equation*}
  \inf_{a \in \mathbb{S}^{d-1}}{}^t\vec{a}S\vec{a}=\lambda_{1}(S)\leq \det(S)^{\frac{1}{d}},
\end{equation*}
where $\lambda_{1}(S)$ is the smallest eigenvalue of $S$.
Thus, it suffices to prove that for every $q\ge 0$ there exist $N\in \mathbb{N}$ and $C>0$ such that for $n\ge N$
\begin{equation*}
  \Prob*{ \inf_{\vec{a}\in\mathbb{S}^{d-1}} \Gamma[\vec{F}_n\cdot \vec{a},\vec{F}_n\cdot \vec{a}]\le \epsilon }\le C\epsilon^q, \qquad \varepsilon > 0.
\end{equation*}

Let $N$ be an integer to be chosen later.
By \cref{discretization}, for every $\vec{a}\in\mathbb{S}^{d-1}$ and $\vec{b}\in\mathbb{S}^{d-1,N}$, we have that $\|\vec{a}-\vec{b}\|\le \frac{C}{N}$.
In this way, 
\begin{equation*}
\abs[\big]{ \Gamma[\vec{F}_n\cdot \vec{a},\vec{F}_n\cdot \vec{a}]-\Gamma[\vec{F}_n\cdot \vec{b},\vec{F}_n\cdot \vec{b}] } = \abs[\big]{ {}^t(\vec{a}-\vec{b})\Gamma(\vec{F}_n)(\vec{a} + \vec{b}) } \leq \frac{2C}{N} \norm[\big]{\Gamma(\vec{F}_n)}.
\end{equation*}
 This gives:
 \begin{equation*}
   \inf_{\vec{a}\in\mathbb{S}^{d-1}} \Gamma[\vec{F}_n\cdot \vec{a},\vec{F}_n\cdot \vec{a}]\le \inf_{\vec{a}\in\mathbb{S}^{d-1,N}} \Gamma[\vec{F}_n\cdot \vec{a},\vec{F}_n\cdot \vec{a}]+\frac{2C}{N}\norm{\Gamma(\vec{F}_n)}.
\end{equation*}
Consequently, we have
\begin{equation*}
\begin{split}
  \left\{\inf_{\vec{a}\in\mathbb{S}^{d-1}} \Gamma[\vec{F}_n\cdot \vec{a},\vec{F}_n\cdot \vec{a}]\le \epsilon\right\} \subset & \left\{\inf_{a\in\mathbb{S}^{d-1,N}} \Gamma[\vec{F}_n\cdot \vec{a},\vec{F}_n\cdot \vec{a}] \le \epsilon+\frac{2C K}{N}\right\}
                                                                                                                            \\& \cup \left\{\|\Gamma(\vec{F}_n)\|>K\right\}.
\end{split}
\end{equation*}

We choose $N= \ceil*{\frac{1}{\epsilon^{2}}}$ and  $K= \ceil*{\frac{1}{\epsilon}}$ so that, for a different constant $C$:
\begin{equation}\label{eq:bound-gamma-sphere}
\left\{\inf_{\vec{a}\in\mathbb{S}^{d-1}} \Gamma[\vec{F}_n\cdot \vec{a},\vec{F}_n\cdot \vec{a}]\le \epsilon\right\}\subset \left\{\inf_{a\in\mathbb{S}^{d-1,N}} \Gamma[\vec{F}_n\cdot \vec{a},\vec{F}_n\cdot \vec{a}] \le C\epsilon\right\}\cup\left\{\|\Gamma(\vec{F}_n)\|>\frac{1}{\epsilon}\right\}.
\end{equation}

By assumption there exists $C>0$ such that for $n$ large,
\begin{equation*}
  \sup_{\vec{a}\in\mathbb{S}^{d-1}}  \Prob*{ \Gamma[\vec{F}_n\cdot\vec{a},\vec{F}_n\cdot\vec{a}]\leq \epsilon }\le C\epsilon^{q+2d}, \qquad \epsilon>0,
\end{equation*}
and we find, with another constant $C' > 0$,
\begin{equation}\label{eq:bound-proba-sphere-1}
  \begin{split}
    \Prob*{ \inf_{a\in\mathbb{S}^{d-1,N}} \Gamma[\vec{F}_n\cdot \vec{a},\vec{F}_n\cdot \vec{a}] \le C\epsilon } & \le \sum_{a \in \mathbb{S}^{d-1,N}} \Prob*{ \Gamma[\vec{F}_n\cdot \vec{a},\vec{F}_n\cdot \vec{a}] \le C\epsilon}
                                                                                                                        \\& \le  CN^{d} \sup_{a\in\mathbb{S}^{d-1}} \Prob*{ \Gamma[\vec{F}_n\cdot \vec{a},\vec{F}_n\cdot \vec{a}] \le C\epsilon }
                                                                                                              \\& \le C' \frac{1}{\epsilon^{2d}}\epsilon^{2d+q}=C'\epsilon^q.
  \end{split}
\end{equation}

Since $(\vec{F}_n)$ is bounded in $L^2$ and $(\vec{F}_{n}) \subset \mathcal{W}_{\leq m}^{d}$, $(\Gamma(\vec{F}_n))$ is also bounded in $L^2$, and then in $L^q$ by equivalence of norms on Wiener chaoses \cref{sec:hypercontractivity}.
Markov inequality gives for some $C$
\begin{equation*}
  \Prob*{ \|\Gamma(\vec{F}_n)\|>\frac{1}{\epsilon} }\le C\epsilon^q.
\end{equation*}
We conclude.
  For completeness, let us sketch the proof of the converse implication \cref{prop:dimupgrade:determinant} $\Rightarrow$ \cref{prop:dimupgrade:sphere}.
  In this case, we start from the bound
  \begin{equation*}
    \det(S) = \prod_{i=1}^{d} \lambda_{i} \leq \lambda_{1} \norm{A}^{d-1}
  \end{equation*}
  Thus,
  \begin{equation*}
    \inf_{\vec{a} \in \mathbb{S}^{d-1}} \vec{a} \cdot \Gamma(\vec{F}_{n}) \vec{a} = \lambda_{1}(\Gamma(\vec{F}_{n})) \geq \det(\Gamma(\vec{F}_{n})) \norm{\Gamma(\vec{F}_{n})}^{-(d-1)}.
  \end{equation*}
  Now, $(\Gamma(\vec{F}_{n}))$ is bounded in all the $L^{p}(\prob)$ ($p\ne \infty$), in view on the assumptions on $(\vec{F}_{n})$ and the equivalence of the norms on $\mathcal{W}_{\leq m}$ (\cref{sec:hypercontractivity}).
\end{proof}

\subsection{Normal approximation in smaller chaos}
In this section, starting from an element of a chaos $\mathcal{W}_m$ whose law is close to a normal law, we construct variables in $\mathcal{W}_{m-1}$ whose laws are also close to normal laws.
This construction allows us to use the induction hypothesis in the proof of \cref{th:main-negative-moments}.
We start with some notations.
\begin{definition}
  Let $F \in \mathbb{R}[N_{1}, \dots, N_{K}]$ be a polynomial.
  If $\vec{x}=(x_1,\ldots,x_K)$ is a vector of $\mathbb{R}^K$, we denote by $\mathsf{D}_{\vec{x}}F$ the \emph{directional derivative} following $\vec{x}$, namely
\begin{equation*}
  \mathsf{D}_{\vec{x}}F \coloneq \sum_{k=1}^K x_k\frac{\partial F_n}{\partial N_k}.
\end{equation*}
If $X$ is a matrix $K\times d$ with column vectors $(\vec{x}_1,\ldots,\vec{x}_d)$, we write
\begin{equation*}
  \mathsf{D}_X F \coloneq (\mathsf{D}_{\vec{x}_1}F,\ldots, \mathsf{D}_{\vec{x}_d} F)
\end{equation*}
\end{definition}
If $F \in \mathcal{W}_m$, then $\mathsf{D}_{\vec{x}}F \in \mathcal{W}_{m-1}$ and $\mathsf{D}_X F \in \mathcal{W}_{m-1}^{d}$.

The following proposition states that if $F\in \mathcal{W}_m$ is close in law to the standard Gaussian $\mathcal{N}(0,1)$, and if we choose $X$ randomly with respect to the Gaussian measure, then $\mathsf{D}_XF \in  \mathcal{W}_{m-1}^d$ is close in distribution to the Gaussian vector $\mathcal{N}\left(0,mI_d\right))$ with large probability.
We write $\gamma_{K,d}$ for the standard Gaussian distribution on the matrices of size $K \times d$.

\begin{proposition}\label{almostCLT}
  Let $m \in \mathbb{N}^{*}$.
  Consider a sequence $(F_n) \subset \mathcal{W}_m^{(0)}$ such that, for all $n \in \mathbb{N}$, $F_{n} \in \mathbb{R}[N_{1}, \dots, N_{K_{n}}]$ for some $K_{n} \in \mathbb{N}^{*}$, and $F_n\to\mathcal{N}(0,1)$ in law.
  Then for every $d \in \mathbb{N}^{*}$ and every $\epsilon>0$,
\begin{equation*}
  \gamma_{K_n,d} \left\{X\in  M_{K_n,d}(\mathbb{R}) : d_{\mathrm{FM}}\paren*{\mathsf{D}_{X}F_n,\mathcal{N}\left(0,mI_d\right)}\ge \epsilon \right\} \xrightarrow[n \to \infty]{} 0.
\end{equation*}
\end{proposition}

  \begin{remark}
    The result states that for $n$ sufficiently large, there exists a set of matrices $X$ of large Gaussian measure, and such that the laws of all the $\mathsf{D}_X F_n$'s are closed to a normal distribution with covariance independent of the $X$.
    It might seem contradictory, since by multiplying $X$ by a scalar $\lambda$ the distribution should change.
    Contrary to the finite dimensional case, the image of an infinite Gaussian measure by a non trivial homothety is singular with respect to the initial measure.
    For instance, by the law of large number the Borel set
    \begin{equation*}
      A \coloneq \set*{ (x_{i}) \in \mathbb{R}^{\mathbb{N}} : \frac{1}{n} \sum_{i=1}^{n} x_{i}^{2} \xrightarrow[n \to \infty]{} 1 },
    \end{equation*}
    has full $\gamma^{\mathbb{N}}$-measure.
    However, the set
    \begin{equation*}
      \lambda A \coloneq \{ \lambda x : x \in A \},
    \end{equation*}
    has measure $0$ as soon as $\abs{\lambda} \ne 1$.
  \end{remark}

\begin{proof}
  Let $(F_{n})$ be as in the theorem.
  Let $(G_{i,j})_{i,j\ge 1}$ be a family of independent standard Gaussian variables, independent of $\vec{N}$.
  For every $n \in \mathbb{N}$, we define the random $K_n\times d$ matrix $\mathscr{G}_{n}$ and the random vector $\vec{V}_{n}$ by
  \begin{align*}
    & \mathscr{G}_n \coloneq (G_{i,j})_{1\leq i\leq d ,1\leq j\leq K_n},
  \\& \vec{V}_n \coloneq \mathsf{D}_{\mathscr{G}_n}F_n=(V_{n,1},\ldots,V_{n,d})
  \end{align*}
   In the Wiener space generated by the variables $(G_{i,j})$ and $(N_k)$,  $\vec{V}_n\in \mathcal{W}_m^d$, and the coordinates read
\begin{equation}\label{formulaV}
  V_{n,i} = \sum_{j=1}^{K_n} \frac{\partial F_n}{\partial N_j}G_{i,j}, \qquad i=1,\dots, d.
\end{equation}
It follows that $\vec{V}_n$ has same law as $\Gamma[F_n,F_n]^{\frac{1}{2}}\vec{G}'$ where $\vec{G}'$ is a standard Gaussian vector independent of $\vec{N}$.
Since by \cref{th:fourth-moment-stein},
\begin{equation}\label{eq:convergence-gamma-fn}
  \Gamma[F_n,F_n] \xrightarrow[n\to +\infty]{L^{2}} m,
\end{equation}
  we deduce that
  \begin{equation*}
    \vec{V}_n \xrightarrow[n \to \infty]{law} \mathcal{N}(0,mI_d).
  \end{equation*}
Using \cref{th:fourth-moment-stein}, we obtain that
\begin{equation}\label{eq:convergence-gamma-vn}
  \Gamma(\vec{V}_n) \xrightarrow[n \to \infty]{L^{2}} m^2 I_d.
\end{equation}

Consider the decomposition $\Gamma(\vec{V}_n) = \Gamma_G(\vec{V}_n)+\Gamma_N(\vec{V_n})$, where $\Gamma_G(\vec{V}_n)$ (resp.\ $\Gamma_N(\vec{V_n})$) is the Malliavin matrix of $\vec{V}_n$ with respect with the coordinates $G_{i,j}$ (resp.\ $N_k$) as defined in \cref{def:gamma-partial-N,def:gamma-partial-G}.
From \cref{formulaV}, we directly compute the matrix $\Gamma_G(\vec{V}_n)$:
\begin{align*}
     & \Gamma_G[V_{n,i},V_{n,j}]=0, \qquad i\not= j;
   \\& \displaystyle \Gamma_G[V_{n,i},V_{n,i}]=\sum_{j=1}^{K_n} \left(\frac{\partial V_{n,i}}{\partial G_{i,j}}\right)^2=\sum_{j=1}^{K_n} \left(\frac{\partial F_{n}}{\partial N_{j}}\right)^2=\Gamma \left[F_{n},F_{n}\right], \qquad i = j.
\end{align*}
Thus, $\Gamma_G(\vec{V}_n)=\Gamma[F_n,F_n] I_d$.
By \cref{eq:convergence-gamma-fn}, we obtain that $\Gamma_G(\vec{V}_n)\to mI_d$ in $L^2$.
Combining with \cref{eq:convergence-gamma-vn}, we deduce that
\begin{equation}\label{eq:convergence-gamma-N-vn}
  \Gamma_N(\vec{V}_n) \xrightarrow[n \to \infty]{L^{2}} m(m-1) I_d.
\end{equation}
Since $\mathsf{D}_XF_n$ depends only on the variables $N_k$'s and not the $G_{i,j}$'s, we get
 \begin{equation*}
   \Gamma(\mathsf{D}_XF_n)=\Gamma_N(\mathsf{D}_XF_n), \qquad \text{for any deterministic}\ X\in M_{K_n,d}(\mathbb{R}).
 \end{equation*}
Thus, we rewrite \cref{eq:convergence-gamma-N-vn} as
\begin{equation}\label{eq:convergence-gamma-DX}
  \int_{M_{K_n,d}(\mathbb{R})} \norm{\Gamma(\mathsf{D}_XF_n)-m(m-1)I_d}_{L^2}^2 \mathrm{d} \gamma_{K_n,d}(X) \xrightarrow[n \to \infty]{} 0
\end{equation}

For $X \in M_{K_{n},d}$, $\mathsf{D}_XF_n \in \mathcal{W}_{m-1}^d$, so that \cref{th:fourth-moment-stein} gives a constant $C=C_m > 0$ such that 
\begin{equation}\label{eq:stein-bound-DX}
  d_{\mathrm{FM}}(\mathsf{D}_XF_n,\mathcal{N}(0,mI_d))^2\leq C \norm{\Gamma(\mathsf{D}_XF_n)-m(m-1)I_d}_{L^2}^{2}.
\end{equation}
Finally, combining \cref{eq:convergence-gamma-DX,eq:stein-bound-DX} yields
 \begin{equation*}
   \int_{M_{K_n,d}(\mathbb{R})}d_{\mathrm{FM}}(\mathsf{D}_XF_n,\mathcal{N}(0,mI_d))^{2}  \mathrm{d}\gamma_{K_n,d}(X)\to 0,
 \end{equation*}
  and we conclude by Markov's inequality.
\end{proof}

\subsection{A compressing argument}\label{s:compressing}
If $F=F(N_1,\ldots,N_K) \in\mathcal{W}_m^{(0)}$, we need to study the $K\times K$ matrix $A \coloneq \nabla^{2}F$.
To that extent, we fix a wisely-chosen $K\times q$ matrix $X$ for a fixed $q$, and we study the $K\times q$ \emph{compressed matrix}  $B \coloneq AX$.
We choose $X$ in a way that $B$ contains most of the information on $A$.
At the same time $B$ is simpler to study since the dimension is reduced.
This strategy appears in information theory under the name of \enquote{compressed sensing}.

\subsubsection{Control of the spectral remainder of the compressed Hessian}
 
An elementary computation shows that the $q\times q$ matrix ${}^tB B = \Gamma(\mathsf{D}_XF)$, and we use the tools developed above in order to study this Malliavin matrix.
  We recall that we have define the spectral remainders of a rectangular matrix $M$ in \cref{sec:spectral} in terms of the singular values and that we have
  \begin{equation*}
    \mathcal{R}_{q}(M) = \mathcal{R}_{q}(({}^{t}M M)^{\frac{1}{2}}), \qquad q \in \mathbb{N}^{*}.
\end{equation*}
\begin{lemma}\label{th:rq-det-Malliavin}
  If $F=F(N_1,\ldots,N_K)$ is polynomial and if $X\in M_{K, q}(\mathbb{R})$, then 
\begin{equation*}
  \mathcal{R}_q(\nabla^{2}FX)=\det(\Gamma(\mathsf{D}_XF)).
\end{equation*}
\end{lemma}

\begin{proof}
  Let $X \coloneq (\vec{x_1},\ldots,\vec{x}_q) = (x_{i,j})_{i\le K, j\le q}$ for some $\vec{x}_{j} \in \mathbb{R}^{K}$ and $B \coloneq (\nabla^{2}F)  X$.
  Then,
\begin{equation*}
  \frac{\partial(\mathsf{D}_{\vec{x}_j} F)}{\partial N_i}=\sum_{k=1}^K x_{k,j}\frac{\partial^2F}{\partial N_i \partial N_k}= B_{i,j}, \qquad i \leq K,\, j \leq q.
\end{equation*}
This shows that $\vec{\nabla}(\mathsf{D}_X F) = B$ and $\Gamma(\mathsf{D}_XF)={}^tB B$.

Moreover, since ${}^tB B$ is a $q\times q$ matrix,  $\mathcal{R}_q(B)$ is by definition the product of the spectral values of  ${}^tB B$, so its determinant. Thus
\begin{equation*}
  \mathcal{R}_q(B)=\det ({}^tB B)=\det(\Gamma(\mathsf{D}_XF)).
\end{equation*}
\end{proof}

\begin{lemma}\label{th:choice-diese-gaussian}
Let $m$, $p$, and $q \in \mathbb{N}^{*}$, with $m \geq 3$.
Assume the induction property $\mathcal{P}(m-1)$.
Then, there exists $C>0$ such that for every $(F_{n}(N_{1}, \dots, N_{K_{n}})) \subset \mathcal{W}_{m}^{(0)}$ converging in law to the standard Gaussian distribution, then, for $n \in \mathbb{N}$, large enough the set 
\begin{equation*}
  \mathcal{E}_{n} \coloneq \left\{X\in M_{K_{n},q}(\mathbb{R}) : \Esp*{ \mathcal{R}_q\left(\nabla^{2}F_{n} X\right)^{-p} }\le C\right\}
\end{equation*}
has $\gamma_{K_{n},q}$-measure more than $\frac{2}{3}$.
\end{lemma}

\begin{proof}
By \cref{cor:dimupgrade}, $\mathcal{P}_q(m-1)$ holds.
In particular, for all $p >0$, there exist $\varepsilon > 0$ and $C>0$ such that for any $\vec{V}\in\mathcal{W}_{m-1}^q$,
\begin{equation}
  d_{\mathrm{FM}}(\vec{V},\mathcal{N}(0,m I_{q}))\le \varepsilon \Rightarrow \Esp*{ \det(\Gamma(\vec{V}))^{-p} }\le C.
\end{equation}
Applying to $\vec{V} \coloneq \mathsf{D}_{X}F_{n}$, we find by \cref{th:rq-det-Malliavin}
\begin{equation*}
  d_{\mathrm{FM}}(\mathsf{D}_{X}F_{n}, \mathcal{N}(0, m I_{q})) \leq \varepsilon \Rightarrow \Esp*{ \mathcal{R}_{q}(\nabla^{2} F_{n} X)^{-p} } \leq C.
\end{equation*}
By \cref{almostCLT}, for $n \in \mathbb{N}$, large enough, the set
\begin{equation*}
  \set*{ X \in M_{K_{n}, q}(\mathbb{R}) : d_{\mathrm{FM}}(\mathsf{D}_{X} F_{n}, \mathcal{N}(0, m I_{q})) \leq \varepsilon },
\end{equation*}
has $\gamma_{K_{n}, q}$-measure more than $2/3$.
This concludes the proof.
\end{proof}

\subsubsection{Relating the spectral remainder of the compressed Hessian and the Hessian}
In this step, we derive estimates on $\mathcal{R}_{q}(A)$ from estimates on $\mathcal{R}_q(AX)$, for a \emph{generic} matrix $X$.

\begin{lemma}\label{th:rq-compressed}
  For every $p,q \in \mathbb{N}^{*}$, there exists $C > 0$ such that for every $d\times d$ symmetric matrix $M$, 
\begin{equation*}
  \Esp*{ \mathcal{R}_q(M\mathcal{X})^p } \leq C \mathcal{R}_q(M)^p.
\end{equation*}
where $\mathcal{X}$ is a $d \times q$ matrix whose entries are independent standard Gaussian variables.
\end{lemma}
\begin{proof}
Let us write $M={}^t P\Delta P$ with $P$ orthogonal and $\Delta$ diagonal, with diagonal values $\lambda_1,\ldots,\lambda_d$.
We have $\mathcal{R}_q(M)=\mathcal{R}_q(\Delta )$.
Also, since $P\mathcal{X}$ and $\mathcal{X}$ have same law, we find that
\begin{equation*}
  {}^{t}(M\mathcal{X})M\mathcal{X} = {}^{t}(P\mathcal{X})\Delta^{2} P\mathcal{X} \overset{\text{Law}}{=} {}^{t}\mathcal{X} \Delta^{2} \mathcal{X}.
\end{equation*}
  In particular, be definition of the spectral remainders for rectangular matrix, $\Esp*{ \mathcal{R}_q(M\mathcal{X})^p }=\Esp*{ \mathcal{R}_q(\Delta \mathcal{X})^p }$.
Thus, we assume that $M=\Delta $.
The entries of $\Delta \mathcal{X}$ are given by $(\Delta \mathcal{X})_{i,j}=\lambda_i \mathcal{X}_{i,j}$.
Thus, for any subsets $I$, $J$ of cardinality $q$ the extracted determinant on $I\times J$ is $\prod_{i\in I}\lambda_i \det(\mathcal{X}_{I,J})$.
By the Cauchy--Binet formula \cref{eq:cauchy-binet},
\begin{equation*}
  \mathcal{R}_q(\Delta \mathcal{X})=\sum_{|I|=q} \prod_{i\in I}\lambda_i^2 S_{I}, \qquad \text{where}\ S_{I} \coloneq \sum_{|J|=q}\det( \mathcal{X}_{I,J})^2.
\end{equation*}
The variables $ S_I$ have same law, and in particular same expectation $c$.
This gives
\begin{equation*}
  \Esp*{ \mathcal{R}_q(\Delta \mathcal{X}) }=c\sum_{|I|=q} \prod_{i\in I}\lambda_i^2=c\mathcal{R}_q(\Delta ).
\end{equation*}
The claim follows for $p=1$.
For $p \ne 1$ we use the equivalence of norms \cref{eq:hypercontractivity}, since $\mathcal{R}_q(\Delta \mathcal{X})$ is a positive polynomial of degree $q$ in Gaussian variables, there exists $C=C_{p,q}$ such that
\begin{equation*}
  \Esp*{ \mathcal{R}_q(\Delta \mathcal{X})^p }\leq C\Esp*{ \mathcal{R}_q(\Delta \mathcal{X}) }^p = C \mathcal{R}_{q}(\Delta)^{p},
\end{equation*}
and the result follows.
\end{proof}

\begin{lemma}\label{th:choice-compressed}
  Let $p$, $q \in \mathbb{N}^{*}$.
  There exists $C>0$ such that for every $K \in \mathbb{N}^{*}$ and every random symmetric matrix $A$ in $ M_{K,K}(\mathbb{R})$, the set
  \begin{equation*}
    \mathcal{E} \coloneq \set*{ X \in M_{K,q}(\mathbb{R}) : \Esp*{ \frac{\mathcal{R}_q (A X)^p}{\mathcal{R}_q(A)^p} }\leq C },
  \end{equation*}
  has $\gamma_{K,q}$-measure more than $\frac{2}{3}$.

\end{lemma}

\begin{proof}
  By \cref{th:rq-compressed}, there exists $C=C_{p,q}$ such that
\begin{equation*}
  \mathcal{R}_q(A)^p\geq  C\int_{M_{K,q}(\mathbb{R})}\mathcal{R}_q(A X)^p \mathrm{d} \gamma_{K,q}(X).
\end{equation*}
Thus,
\begin{equation*}
  \int_{M_{K,q}(\mathbb{R})} \Esp*{ \frac{\mathcal{R}_q (A X)^p}{\mathcal{R}_q(A)^p} } \mathrm{d} \gamma_{K,q}(X) \le \Esp*{ \frac{1}{\mathcal{R}_q(A)^p}\times C \mathcal{R}_q(A)^p }=C.
\end{equation*}
As a result, we obtain, by Markov inequality, that
\begin{equation}\label{condition1}
  \begin{split}
  & \gamma_{K,q} \set*{ X \in M_{K,q}(\mathbb{R}) : \Esp*{ \frac{\mathcal{R}_q (A X)^p}{\mathcal{R}_q(A)^p} } \ge 3 C  }
\\& \le \frac{1}{3 C}\int_{M_{K,q}(\mathbb{R})} \Esp*{ \frac{\mathcal{R}_q (A X)^p}{\mathcal{R}_q(A)^p} } \mathrm{d} \gamma_{K,q}(X)
\\& \le \frac{1}{3}.
  \end{split}
\end{equation}
The proof is complete.
\end{proof}

\subsection{Proof of the induction step}\label{s:induction}
\begin{proof}[Proof of {\cref{th:main-negative-moments}}]
We establish the induction step $\mathcal{P}(m-1)\Rightarrow \mathcal{P}(m)$.
Let $m \in \mathbb{N}$, $m \geq 3$, and assume $\mathcal{P}(m -1)$.
We fix a sequence $(F_n) \subset \mathcal{W}_m^{(0)}$ converging in law to $\mathcal{N}(0,1)$.
As before, we assume that $F_n \in \mathbb{R}[N_1,\ldots, N_{K_n}]$, and we set $A_n \coloneq \nabla^{2}F_{n}$, which is a random matrix of size $K_n\times K_n$.
Fix $p$ and $q \in \mathbb{N}^{*}$, and fix $n \in \mathbb{N}$ large enough for \cref{th:choice-diese-gaussian} to apply to $F_n$.
Thus, there exists $\mathcal{E}_1\subset M_{K_{n},q}(\mathbb{R})$ of $\gamma_{K_{n},q}$-measure more than $\frac{2}{3}$ such that
\begin{equation}\label{eq:bound-rq-compressed}
  \Esp*{ \mathcal{R}_q\left( A_n X\right)^{-p} }\le C_1, \qquad X \in \mathcal{E}_{1},
\end{equation}
where $C_1 > 0$ depends only on $m$, $p$ and $q$.
By \cref{th:choice-compressed}, there exists $\mathcal{E}_2\subset M_{K_n,q}(\mathbb{R})$ of $\gamma_{K_n,q}$-measure more than $\frac{2}{3}$ such that
\begin{equation}\label{eq:bound-ratio-rq}
  \Esp*{ \frac{\mathcal{R}_q (A_n X)^p}{\mathcal{R}_q(A)^p} }\leq C_2, \qquad X \in \mathcal{E}_{2},
\end{equation}
where $C_2$ depends only on $p$ and $q$.
Since $\gamma_{K_n,q}\left(\mathcal{E}_1\cap \mathcal{E}_2\right)\ge \frac{1}{3}$, the sets $\mathcal{E}_1$ and  $\mathcal{E}_2$ have a non empty intersection.
In particular, there exists $X\in M_{K_n,q}(\mathbb{R})$ such that estimates \cref{eq:bound-rq-compressed,eq:bound-ratio-rq} hold simultaneously.
Then by Cauchy--Schwarz inequality
\begin{equation*}
  \Esp*{ \mathcal{R}_q(A_n)^{-\frac{p}{2}} }^{2} \leq \Esp*{ \mathcal{R}_q\left( A_n X\right)^{-p} }\Esp*{ \frac{\mathcal{R}_q (A_n X)^p}{\mathcal{R}_q(A_n)^p} }\leq C_1 C_2.
\end{equation*}
Since in the previous argument, $p$ and $q$ are arbitrary positive integers, we have shown
\begin{equation*}
  \limsup_{n\to +\infty} \Esp*{ \mathcal{R}_q(A_n)^{-p} }<+\infty, \qquad p,\, q \in \mathbb{N}^{*}.
\end{equation*}
Specifying the above estimate to $p=\frac{1}{2}$, we deduce from \cref{th:bound-gamma-spectral-hessian} that
\begin{equation*}
  \limsup_{n\to +\infty} \Esp*{ \Gamma[F_n,F_n]^{-q} } < +\infty, \qquad q \in \mathbb{N}.
\end{equation*}
This shows $\mathcal{P}(m)$.
This completes the induction step, and thus the proof of \cref{th:main-negative-moments}.
\end{proof}

\section{Multivariate random variables and sums of chaoses}

In this section, we prove \cref{th:main-negative-moments-sum-chaos,th:main-negative-moments-vector}.

\subsection{A central limit theorem for iterated sharp operators}
We recall that the iterated sharp operators are defined on polynomials $F\in \mathbb{R}[N_1,\ldots, N_{K}]$ by
\begin{equation*}
  \sharp^k[F] \coloneq \sum_{1\le i_1,\ldots,i_k}\frac{\partial^k F}{\partial N_{i_1}\cdots \partial N_{i_k}}(\vec{N}) G_{1,i_1}\cdots G_{k,i_k},
\end{equation*}
where $(G_{i,j})$ is a family of independent standard Gaussian independent of $N_k$.
We prove that on Wiener chaoses, the property of converging to a Gaussian distribution is preserved by applications of iterated sharp.

\begin{proposition}\label{supersharpCLT}
  For any sequence  $(F_n)_{n\in \mathbb{N}}$ in $\mathcal{W}_m$, we have
  \begin{equation*}
    F_n \xrightarrow[n\to\infty]{\text{Law}}~\mathcal{N}\left(0,1\right)\Rightarrow \sharp^k[F_n] \xrightarrow[n\to\infty]{\text{Law}}~\mathcal{N}\left(0,\sigma^2\right),
  \end{equation*}
  where $\sigma^2=m(m-1)\cdots (m-k+1)$.
\end{proposition}

We prove the following lemma.
  We write $\gamma \coloneq \mathcal{N}(0,1)$, and $\gamma^{\mathbb{N}} \coloneq \otimes_{k \in \mathbb{N}} \gamma$.
\begin{lemma}\label{supersharpCLTinduc}
  Let $(F_n)$ a sequence in $\mathcal{W}_m^{(0)}$ such that $F_n \xrightarrow[n\to\infty]{\text{Law}}~\mathcal{N}\left(0,1\right)$.
  Then there exists a subsequence $(F_{\phi(n)})$ such that for $\gamma^{\mathbb{N}}$-almost every sequence $(x_i)_i$,
  \begin{equation*}
    \sum_{i} \frac{\partial F_{\phi(n)}}{\partial N_i}(\vec{N}) x_i  \xrightarrow[n\to\infty]{\text{Law}}~\mathcal{N}\left(0,m\right).
  \end{equation*}
\end{lemma}

\begin{proof}
  For an infinite vector $\vec{x} \coloneq (x_1,x_2,\ldots)$, we set  
  \begin{equation*}
    \mathsf{D}_{\vec{x}}F_n \coloneq \sum_{i} \frac{\partial F_n}{\partial N_i}(\vec{N}) x_i,
  \end{equation*}
  where the sum is finite since $F_n\in \mathcal{W}_m^{(0)}$.
  By \cref{almostCLT} with $d=1$, we deduce that the sequence of measurable mappings $\vec{x}\mapsto d_{\mathrm{FM}}(\mathsf{D}_{\vec{x}}F_n, \mathcal{N}\left(0,m\right))$ tends to $0$ in probability on $(\mathbb{R}^\mathbb{N},\gamma^\mathbb{N})$.
  Thus, there exists a subsequence which converges for $\gamma^\mathbb{N}$-almost every vector $\vec{x}$.
  The result follows.
\end{proof}

 \begin{proof}[Proof of {\cref{supersharpCLT}}]
   By successive applications of \cref{supersharpCLTinduc}, there exists a subsequence $(F_{\phi(n)})$ such that for $\gamma^\mathbb{N}\otimes\cdots \otimes  \gamma^\mathbb{N}$-almost every sequences $(x_{1,i})_i,\cdots ,(x_{k,i})_i$,
   \begin{equation*}
     \sum_{1\le i_1,\ldots,i_k}\frac{\partial^k F_{\phi(n)}}{\partial N_{i_1}\cdots \partial N_{i_k}}(\vec{N}) x_{1,i_1}\cdots x_{k,i_k} \xrightarrow[n\to\infty]{\text{Law}}~\mathcal{N}\left(0,\sigma^2\right),
   \end{equation*}
where $\sigma^2=m(m-1)\cdots (m-k+1)$.
  Take a continuous and bounded function $h \colon \mathbb{R} \to \mathbb{R}$.
  By the previous convergence, we find that for $\gamma^\mathbb{N}\otimes\cdots \otimes  \gamma^\mathbb{N}$-almost every sequences $(x_{1,i})_i,\cdots ,(x_{k,i})_i$,
  \begin{equation*}
    \Esp*{ h\paren*{\sum_{1\le i_1,\ldots,i_k}\frac{\partial^k F_{\phi(n)}}{\partial N_{i_1}\cdots \partial N_{i_k}}(\vec{N}) x_{1,i_1}\cdots x_{k,i_k} }} \xrightarrow[n \to \infty]{} \int h(\sigma x) \gamma(\mathrm{d} x).
  \end{equation*}
  Integrating each of $x_{i,j}$'s with respect to $\gamma^{\mathbb{N}}$, we obtain, by dominated convergence, that
\begin{equation*}
  \Esp*{ h\paren*{\sharp^{k}[F_{\varphi(n)}]}} \xrightarrow[n \to \infty]{} \int h(\sigma x) \gamma(\mathrm{d} x).
\end{equation*}
This gives convergence in law of the subsequence.
Since this reasoning applies on every subsequence of $(F_n)$, we conclude on the convergence in law of the full sequence.
\end{proof}

\subsection{Proofs of the remaining theorems}
The following statement is slightly more precise than \cref{th:main-negative-moments-sum-chaos}; we use it for the proof of \cref{th:main-negative-moments-vector}.
Recall that we write $\mathsf{J}_{m}F$ for the projection of $F$ on the $m$-th Wiener chaos.
\begin{proposition}\label{sumofchaosquanti}
  Let $m$ and $q\in \mathbb{N}$.
  There exist $\delta>0$, $r>0$ and $C>0$ such that the following statement holds: for every $F\in\mathcal{W}_{\leq m}$
  \begin{equation*}
    \bracket*{d_{\mathrm{FM}}(\mathsf{J}_{m}F,\mathcal{N}(0,1))\le \delta} \Rightarrow \bracket*{\Esp*{ \Gamma\left[F,F\right]^{-q} }\le C \norm{F}_{L^2}^{r}}.
  \end{equation*}
\end{proposition}

\begin{proof}
  Let $F \in \mathcal{W}_{\leq m}$ and $\widetilde{F} \coloneq \mathsf{J}_{m}F$.
 By definition of the sharp operators, $\sharp^m[\widetilde{F}]=\sharp^m[F]$.
 Let $q \in \mathbb{N}$.
 Applying \cref{supersharpnegmoments} to $\norm{F}_{L^{2}}^{-1}F$, there exist $q' \in \mathbb{N}$ and $C > 0$ such that
\begin{equation*}
  \Esp*{ \Gamma\left[F,F\right]^{-q} }\le C\norm{F}_{L^2}^{r} \Esp*{ \Gamma\left[\sharp^m[F],\sharp^m[F]\right]^{-q'} }
\end{equation*}
where $r \coloneq 2(q'-q)$.
By \cref{supersharpCLT}, for every $\delta'>0$ there exists $\delta>0$ such that
\begin{equation*}
  d_{\mathrm{FM}}(\widetilde{F},\mathcal{N}(0,1))\le \delta \Rightarrow d_{\mathrm{FM}}(\sharp^m[\widetilde{F}],\mathcal{N}(0,\sigma^{2}))\le \delta'.
\end{equation*}
where $\sigma \coloneq m!^{\frac{1}{2}}$.
By \cref{th:main-negative-moments}, for every $q' \in \mathbb{N}$ there exist $\delta' > 0$ and $C > 0$ such that
\begin{equation*}
  d_{\mathrm{FM}}(\sharp^m[F],\mathcal{N}(0,\sigma^{2}))\le \delta'\Rightarrow \Esp*{ \Gamma\left[\sharp^m[F],\sharp^m[F]\right]^{-q'} }\le C .
\end{equation*}

Combining the three estimates above gives the result.

\end{proof}

\cref{th:main-negative-moments-sum-chaos} follows immediately.
Now we prove \cref{th:main-negative-moments-vector}.
In order to use \cref{prop:dimupgrade}, we prove the following lemma.

\begin{lemma}\label{th:negative-moments-sphere}
  Let $d \in \mathbb{N}^{*}$, and $m_1,\cdots , m_d \in \mathbb{N}^{*}$ and $q > 0$.
  There exist $\delta>0$ and $C>0$ such that the following statement holds:
  for every $\vec{F}=(F_1,\ldots,F_d) \in \mathcal{W}_{m_1} \times \ldots \times\mathcal{W}_{m_d}$ such that $d_{\mathrm{FM}}(\vec{F},\mathcal{N}(0,I_d))\le \delta$, for every $\vec{a}\in \mathbb{S}^{d-1}$, the variable $F_{\vec{a}}=\sum_{i=1}^{d} a_i F_i=\vec{F}\cdot\vec{a}$ satisfies
\begin{equation*}
  \Esp*{ \Gamma[F_{\vec{a}},F_{\vec{a}}]^{-q} }\le C.
\end{equation*}
\end{lemma}

\begin{proof}
We proceed by induction on $d$.
The case $d=1$ is \cref{th:main-negative-moments}.
We fix $d\ge 2$,  $m_1,\cdots , m_d$, $F_1,\ldots,F_d$ and $\vec{a}$ as in the statement.
We set $m \coloneq \max_i m_i$.
For $\epsilon>0$ we bound $\Prob*{ \Gamma[F_{\vec{a}},F_{\vec{a}}]\le \epsilon }$ in two different ways according to the relative size of $|a_d|$ compared to $\epsilon$.
Fix $q \in \mathbb{N}$ and $\epsilon\in (0,1/2)$, and set $\alpha \coloneq \min\left(\frac{q}{2r},1\right)$ where $r$ is given by \cref{sumofchaosquanti}.
\begin{itemize}[wide]
\item Assume $|a_d|\geq \epsilon^{\alpha}$.
Then, \cref{sumofchaosquanti} applied to $\frac{1}{a_d}F_{\vec{a}}$ implies that if $d_{\mathrm{FM}}(F_d,\mathcal{N}(0,1))$ is small enough then
\begin{equation*}
  \Esp*{ \Gamma\left[F_{\vec{a}},F_{\vec{a}}\right]^{-q} }\le C|a_d|^{-r}
\end{equation*}
for a constant $C=C_{d,m,q}$.
Thus,
\begin{equation*}
  \Prob*{ \Gamma[F_{\vec{a}},F_{\vec{a}}]\le \epsilon }\le C|a_d|^{-r}\epsilon^q\leq C\epsilon^{\frac{q}{2}}.
\end{equation*}
\item Assume $|a_d|\leq \epsilon^{\alpha}$.
  Define
  \begin{equation*}
F_{\vec{a}}' \coloneq \sum_{i=1}^{d-1} a_i F_i.
\end{equation*}
By using the induction hypothesis with some number $Q$ to be chosen later, we find a constant $C=C_{d,m,Q} > 0$ such that if $d_{\mathrm{FM}}(\vec{F}',\mathcal{N}(0,I_{d-1}))$ is small enough then
\begin{equation*}
  \Esp*{ \Gamma\left[F_{\vec{a}}',F_{\vec{a}}'\right]^{-Q} }\le C
\end{equation*}
By hypercontractivity \cref{eq:hypercontractivity}, we find another constant $C=C_{m} > 0$ such that 
\begin{equation*}
  \norm{ \Gamma[F_{\vec{a}}',F_{\vec{a}}']-\Gamma[F_{\vec{a}},F_{\vec{a}}] }_{L^Q}\leq C|a_d|\leq C\epsilon^\alpha.
\end{equation*}
Then, using that $\epsilon\leq \frac{1}{2}\epsilon^{\frac{\alpha}{2}}$ (since $\alpha\leq 1$ and $\epsilon<1/2$), we write
\begin{equation*}
\Prob*{ \Gamma[F_{\vec{a}},F_{\vec{a}}]\le \epsilon }\leq \Prob[\Big]{ \Gamma[F_{\vec{a}}',F_{\vec{a}}']\le \epsilon^{\frac{\alpha}{2}}  } + \Prob[\Big]{ \abs{ \Gamma[F_{\vec{a}}',F_{\vec{a}}']-\Gamma[F_{\vec{a}},F_{\vec{a}}] } \ge \tfrac{1}{2}\epsilon^{\frac{\alpha}{2}}  }.
\end{equation*}
where
\begin{align*}
  & \Prob*{ \Gamma[F_{\vec{a}}',F_{\vec{a}}']\le \epsilon^{\frac{\alpha}{2}} }\leq C\epsilon^{Q\frac{\alpha}{2}}
\\& \Prob*{ \abs{ \Gamma[F_{\vec{a}}',F_{\vec{a}}']-\Gamma[F_{\vec{a}},F_{\vec{a}}] } \ge \tfrac{1}{2}\epsilon^{\frac{\alpha}{2}} } \leq 
\paren*{ \frac{2 \norm{ \Gamma[F_{\vec{a}}',F_{\vec{a}}']-\Gamma[F_{\vec{a}},F_{\vec{a}}] }_{L^Q}}{\epsilon^{\frac{\alpha}{2}}} }^Q\leq C\epsilon^{Q\frac{\alpha}{2}}
\end{align*}
Choosing $Q$ such that $Q\alpha\ge q$, for instance, $Q \geq 4r$, we deduce that there exists $C=C_{d,m,q}$ such that 
\begin{equation*}
  \Prob*{ \Gamma[F_{\vec{a}},F_{\vec{a}}]\le \epsilon }\leq C\epsilon^{\frac{q}{2}}.
\end{equation*}
\end{itemize}

Combining the two cases, we obtained that for every $q\ge 0$ if $d_{\mathrm{FM}}(\vec{F},\mathcal{N}(0,I_d))$ is small enough there exists $C=C_{d,m,q}$ such that 
\begin{equation*}
   \Prob*{ \Gamma[F_{\vec{a}},F_{\vec{a}}]\le \epsilon }\le C\epsilon^{\frac{q}{2}}, \qquad \varepsilon > 0.
\end{equation*}
The conclusion follows.
\end{proof}

\begin{proof}[Proof of \cref{th:main-negative-moments-vector}].
  Let $m_1,\ldots, m_d \in \mathbb{N}^{*}$.
  Consider a sequence $(\vec{F}_n) \subset \mathcal{W}_{m_1}\times\cdots \times \mathcal{W}_{m_d}$ converging in law to $\mathcal{N}(0,I_d)$, and a sequence $(\vec{a}_n) \subset \mathbb{S}^{d-1}$.
  For every $q\ge 0$, we apply \cref{th:negative-moments-sphere} to $\vec{F}_n$, and we deduce that there exists $N \in \mathbb{N}$ such that
\begin{equation*}
  \sup_{n\geq N}\Esp*{ \Gamma[\vec{F}_n\cdot \vec{a}_n, \vec{F}_n\cdot \vec{a}_n]^{-q} }<+\infty, \qquad n \geq N.
\end{equation*}
Then, by \cref{prop:dimupgrade}, for every $q\ge 0$ there exists $N \in \mathbb{N}$ such that
\begin{equation*}
   \Esp*{ \det\Gamma (\vec{F}_n)^{-q}  }\leq C, \qquad n \geq N.
\end{equation*}
\end{proof}

\begin{acks}[Acknowledgments]
The authors are grateful to G.\ Cébron for suggesting the reference \cite{BercoviciVoiculescu} regarding superconvergence in the free central limit theorem.
  The authors are indebted to the anonymous referee.
  Their thorough review and thoughtful comments have greatly helped us in the preparation of the final version of this paper.
\end{acks}
\begin{funding}
R.H.\ gratefully acknowledges funding from Centre Henri Lebesgue (ANR-11-LABX-0020-01) through a research fellowship in the framework of the France 2030 program.
This work was supported by the ANR grant UNIRANDOM, (ANR-17-CE40-0008).
\end{funding}

\bibliographystyle{imsart-number.bst} 
\bibliography{super_fmt.bib}       

\end{document}